\documentclass[a4paper,10pt]{amsart}

\usepackage{macros}

\begin{document}
\title[Metric completions from finite dimensional algebras]{Metric completions of triangulated categories from finite dimensional algebras}

\author[\textsc{Cyril Matou\v sek}]{\textsc{Cyril Matou\v sek} \orcidlink{0009-0005-0299-7021}}
\address{Aarhus University,
Department of Mathematics,
Ny Munkegade 118,
8000 Aar\-hus~C,
Denmark}
\email{cyril.matousek@math.au.dk}

\thanks{This work was supported by a research grant (VIL42076) from VILLUM FONDEN}
\date{\today}

\begin{abstract}
In this paper, we study metric completions of triangulated categories in~a~re\-pres\-en\-ta\-tion-the\-o\-re\-tic context.
We provide a concrete description of completions of bounded derived categories of hereditary finite dimensional algebras of finite representation type.
In order to investigate completions of~bounded derived categories of algebras of finite global dimension, we define image and preimage metrics under a triangulated functor and use them to~induce a triangulated equivalence between two completions.
Furthermore, for a given metric on a~triangulated category we construct a new, closely related good metric called the improvement and compare the respective completions.
\end{abstract} 

\maketitle
\vspace{4ex}

{
\hypersetup{linkcolor=black}
\tableofcontents
}

\section{Introduction}
\label{sec:intro}

Since the introduction of triangulated categories by Verdier \cite{Verdier96} and Puppe \cite{Puppe62}, there have not been many known ways to construct a new triangulated category from an~already existing one without the presence of an enhancement or a separable monoidal structure \cite{Balmer11}.
Together with other rare, advanced methods such as orbit categories \cite{Keller05} and relative stable categories \cite{Beligiannis00}, \cite{BalmerStevenson21}, Neeman's recent groundbreaking result from \cite{Neeman18} goes against the prevailing belief that triangulated categories do not reproduce.

In \cite{Neeman18}, Neeman constructs completions of triangulated categories using metric techniques. His method resembles the~construction of a completion of~a~metric space, while the concept of defining a metric (or a norm) on a category evolves from the influential paper \cite{Lawvere73} by Lawvere from 1973. Neeman also employs his previous ideas about approximability of triangulated categories from~\cite{Neeman21}. By defining a~metric on a triangulated category, we assign rational-valued lengths to~morphisms in~the~category, and then formally add colimits of Cauchy sequences (i.e.\ \(\N\)-shaped diagrams where morphisms eventually get shorter and shorter) to our category. If we then restrict ourselves to so-called compactly supported elements, i.e.\ those colimits of~Cauchy sequences which behave nicely with respect to very short morphisms, we obtain a~triangulated category called the completion.

Neeman's original motivation in \cite{Neeman18} was to generalise Rickard's results on derived Morita equivalences from~\cite{Rickard89}.
He shows that for a noetherian and weakly approximable triangulated category \(\mathcal{T}\) there exist intrinsically determined metrics on the triangulated categories of compact elements \(\mathcal{T}^c\) of \(\mathcal{T}\) and the bounded category \(\mathcal{T}^b_c\) (in the sense of \cite[Definition~0.16]{Neeman21}).
Using these canonical metrics, one can show that \(\mathcal{T}^c\) and \(\mathcal{T}^b_c\) are completions of each other (up to switching to opposite categories). For a coherent ring $R$, setting \(\mathcal{T}:=\derived(\Mod\dashmodule R)\)
recovers Rickard's original result \cite[Proposition~8.2]{Rickard89}. 
Another particular instance of~such a phenomenon is the case of a derived category \(\derived(X)\) of~a~separated, noetherian scheme $X$ and the subcategories \(\derived^{\text{perf}}(X)\) (perfect complexes) and \(\derived^b_{\text{coh}}(X)\) (complexes with bounded, coherent cohomology).
For~more examples and applications, we refer to \cite{Neeman18}, \cite{Neeman22} and Section 5 of \cite{CanonacoNeemanStellari24}. The last two aforementioned papers elaborate further on the method of metric completions in~the~context of approximable triangulated categories.

Neeman summarizes his method of metric completions in his survey paper \cite{Neeman20} where he also introduces a more refined version of a metric, a so-called good metric. While the definition of a good metric is more restrictive, it makes statements of~theorems and their proofs slicker and less technical. We shall provide more details on good metrics in Section \ref{sec:improvements}.

Even though Neeman's result on completing triangulated categories with respect to a metric is relatively new, metric completions have already found their use in~several recent papers. 
Metric completions have been used by Sun and Zhang in~\cite{SunZhang21} to study recollements of triangulated categories, and their main theorem \cite[Theorem~4.4]{SunZhang21} states that under certain conditions the construction of a completion transforms a recollement of triangulated categories into a right recollement of~the~corresponding completions.
In \cite{BiswasChenRahulParkerZheng24}, Biswas, Chen, Rahul, Parker, and Zheng use completions (and invariance under the completion construction) with~respect to~certain natural good metrics on a triangulated category as a criterion for~existence (and mutual equivalence) of bounded $t$-structures on said triangulated category (\cite[Theorem~1.5]{BiswasChenRahulParkerZheng24}).
Cummings and Gratz calculate completions of discrete cluster categories of type \(\mathbb{A}\) with respect to good metrics in \cite{CummingsGratz24}, and give their combinatorial descriptions.

In this paper, we study metrics and metric completions in a~representation-the\-o\-re\-tic setting. While the results from \cite{Neeman18} and \cite{BiswasChenRahulParkerZheng24} focus mainly on metrics coming from~$t$-structures with motivation in algebraic geometry, our aim is to systematically calculate completions with respect to any arbitrary metric on a bounded derived category \(\derived^b(\modf\dashmodule
A)\) of a sufficiently nice finite dimensional algebra $A$.

We provide an explicit formula for calculating completions of \(\derived^b(\modf\dashmodule
A)\) for $A$ a~hereditary algebra of finite representation type (e.g.\ a path algebra of a simply laced Dynkin quiver) with respect to any metric.
A metric on a triangulated category formally consists of a decreasing chain of subcategories (see Definition~\ref{DefinitionMetric}) that should be thought of as shrinking open balls around zero, so in our case we can express the~completion in terms of these open balls.

\begin{theorem}[{Theorem~\ref{DynkinCompletions}}]
\label{IntroductionDynkinCompletions}
Let \(A\) be a hereditary finite dimensional algebra (over a field) of finite representation type.
Let \(\{B_n\}_{n\in\N}\) be a good metric on~the~category \(\derived^b(\modf\dashmodule A)\).
Denote \(\mathcal{B}:=\bigcap_{n=1}^\infty B_n\).

The completion of \(\derived^b(\modf\dashmodule A)\) with respect to the good metric \(\{B_n\}_{n\in\N}\) is equal to
\(
\mathcal{B}^\perp
\subseteq\derived^b(\modf\dashmodule A)
\)
where
\(\mathcal{B}^\perp:=\bigcap_{B\in\mathcal{B}}\Ker\Hom_{\derived^b(\modf\dashmodule A)}(B,\blank)\).
\end{theorem}

The condition on \(\{B_n\}_{n\in\N}\) of being a good metric is, in fact, not necessary. Theorem~\ref{DynkinCompletions} is stated more generally and encompasses also the case where the~metric is not good. 
A noteworthy corollary of our theorem is that the thick subcategories of \(\derived^b(\modf\dashmodule A)\) are precisely the completions of \(\derived^b(\modf\dashmodule A)\) with respect to some metric (see Corollary~\ref{ThickSubcategories}).

Our second main result is concerned with functors between completions of triangulated categories of the form \(\derived^b(\Mod\dashmodule A)\) for a finite dimensional algebra of~finite global dimension, or even more generally between completions of~$K$\=/linear triangulated categories (over a field $K$) with Serre functors. 
Our starting point is Theorem~4.3 from Sun and Zhang’s paper \cite{SunZhang21}, that triangulated functors between triangulated categories under certain continuity conditions induce triangulated functors on their completions, whose generalisation for general (non-good) metrics is presented in this paper.

\begin{theorem}[{\cite[Theorem~4.3, pt.\ 3]{SunZhang21}}, Theorem~\ref{SunZhangTheorem}]
\label{IntroductionSunZhangTheorem}
Let \(F:\mathcal{T}\rightleftarrows \mathcal{S}:G\) be adjoint triangulated functors between triangulated categories equipped with metrics. 
If $F$ and $G$ are compressions,
then the precomposition \(\blank\circ F:\Mod\dashmodule\mathcal{S}\rightarrow\Mod\dashmodule\mathcal{T}\) induces a~triangulated functor \(\blank\circ F:\mathfrak{S}({\mathcal{S}})\rightarrow\mathfrak{S}(\mathcal{T})\) between the completion \(\mathfrak{S}({\mathcal{S}})\) of \(\mathcal{S}\) and the completion \(\mathfrak{S}({\mathcal{T}})\) of \(\mathcal{T}\).

In other words, we have a commutative diagram
 \[
\begin{tikzcd}
    \mathfrak{S}({\mathcal{S}}) \arrow[d, hook] \arrow{r}{\blank\circ F\restriction\mathfrak{S}({\mathcal{S}})} & \mathfrak{S}({\mathcal{T}}) \arrow[d, hook] \\
    \Mod\dashmodule\mathcal{S} \arrow{r}{\blank\circ F}                & \Mod\dashmodule\mathcal{T}               
\end{tikzcd}
\]
with the upper horizontal functor \(\blank\circ F\restriction\mathfrak{S}({\mathcal{S}})\) being triangulated and the vertical functors being inclusions.
\end{theorem}

Utilising Theorem~\ref{IntroductionSunZhangTheorem}, we can express a completion of triangulated category as a~completion of another using a related metric. For that purpose we define and study the image and~the~preimage of a metric under~a~triangulated functor. It turns out that the preimage of a metric is always a metric (see Lemma~\ref{PreimageIsAMetric}) and the image of a metric is a metric if it is taken under a full triangulated functor (see Lemma~\ref{ImageIsSometimesAMetric}).
In the following theorem we use a preimage metric in such a way, that all relevant functors are compressions and consequently Theorem~\ref{IntroductionSunZhangTheorem} applies.

\begin{theorem}[Theorem~\ref{CalculatingCompletionsElsewhere}]
\label{IntroductionCalculatingCompletionsElsewhere}
Consider a diagram of triangulated functors 
\vspace{-0.1cm}
\begin{equation}\nonumber   
\xymatrix@C=0.5cm{\mathcal{S} \ar[rrr]^{G}     &&& \mathcal{T}\ar @/_1.5pc/[lll]_{F} \ar @/^1.5pc/[lll]_{J} } 
\end{equation}
consisting of two adjoint pairs \(F:\mathcal{T}\rightleftarrows \mathcal{S}:G\) and \(G:\mathcal{S}\rightleftarrows \mathcal{T}:J\). 
Let \(\{B_n\}_{n\in\N}\) be a metric on \(\mathcal{T}\).
Denote the~completion of~\(\mathcal{T}\) with respect to \(\{B_n\}_{n\in\N}\) as \(\mathfrak{S}(\mathcal{T})\) and
the completion of \(\mathcal{S}\) with respect to~the~preimage metric \(\{G^{-1}(B_n)\}_{n\in\N}\) as \(\mathfrak{S}(\mathcal{S})\).

If $F$ and $J$ are fully faithful, then the induced functor
\(\blank\circ G:\mathfrak{S}({\mathcal{T}})\rightarrow\mathfrak{S}(\mathcal{S})\) given by Theorem~\ref{IntroductionSunZhangTheorem} is a well-defined triangulated equivalence with a triangulated inverse \(\blank\circ F\).
\end{theorem}

Many of recent articles (\cite{Neeman20}, \cite{SunZhang21}, \cite{BiswasChenRahulParkerZheng24}, \cite{CummingsGratz24}) about metric completions of~triangulated categories work with good metrics only. We finish this paper by a~short comparison of generic metrics and good metrics in Section~\ref{sec:improvements}. We show that every metric has a naturally given improvement which is always a good metric (see Definition~\ref{DefinitionImprovement} and Proposition~\ref{ImprovementIsNotRandom}). A completion of a triangulated category with respect to a given metric then embeds into the completion with respect to the~original metric's improvement as a triangulated subcategory (see Corollary~\ref{ImprovementContainsOriginalCompletion}). 

\subsubsection*{Acknowledgements}
I wish to express my gratitude to my supervisor, Sira Gratz, for introducing me to the topic of metrics on triangulated categories and for her valuable comments regarding the paper. I thank Yaohua Zhang for making me aware of the~content of Remark~\ref{AnAlternativeProof}. I would also like to thank everyone from the ``Categories, Clusters and Completions'' and ``Homological Algebra'' groups at~Aarhus University.
This work was supported by a research grant (VIL42076) from VILLUM FONDEN.

\section{Preliminaries}
\label{sec:preliminaries}

In this section, we recall definitions, theorems and notation related to completions of triangulated categories via metric techniques introduced by Neeman in \cite{Neeman18}, and his follow up survey paper \cite{Neeman20}.

\subsection{Notation and conventions}

\subsubsection{General category theory}
\label{GeneralCategoryTheory}

If we have an adjoint pair \(F,G\) of functors between two categories \(\mathcal{A},\mathcal{B}\), i.e.\ \(F:\mathcal{A}\rightarrow\mathcal{B}\) and \(G:\mathcal{B}\rightarrow\mathcal{A}\), we write it as \(F:\mathcal{A}\rightleftarrows \mathcal{B}:G\). By this notation, we mean that $F$ is the left adjoint of $G$, and $G$ is the right adjoint of~$F$.

For \(\mathcal{C}\) a category, we will simply denote \(C\in\mathcal{C}\) (by an abuse of notation) as a~shortcut for \(C\in\Ob(\mathcal{C})\). All subcategories are considered full. For \(\mathcal{A},\mathcal{B}\) categories, we use the notation \(\mathcal{A}\subseteq\mathcal{B}\) for~denoting that \(\mathcal{A}\) is a subcategory of \(\mathcal{B}\).

Given a natural transformation \(\varphi:G\Rightarrow H\) between two functors \(G,H:\mathcal{B}\rightarrow \mathcal{C}\) with components \(\varphi_B:G(B)\rightarrow H(B)\) for \(B\in\mathcal{B}\) and a functor \(F:\mathcal{A}\rightarrow\mathcal{B}\), we shall denote \(\varphi_{F}:G\circ F\Rightarrow H\circ F\) the natural transformation with components \(\varphi_{F(A)}:GF(A)\rightarrow HF(A)\) for each \(A\in\mathcal{A}\).

\(\Ab\) denotes the category of abelian groups.

\subsubsection{Additive categories}

 Let \(\mathcal{A}\) be an additive category. We denote the category of additive contravariant functors from \(\mathcal{A}\) to \(\Ab\) as \(\Mod\dashmodule \mathcal{A}\), and \(Y:\mathcal{A}\rightarrow\Mod\dashmodule\mathcal{A}\) the Yoneda embedding \(A\mapsto\Hom_{\mathcal{A}}(\blank,A)\).

For a class of objects \(\mathcal{C}\subseteq\Ob(\mathcal{A})\) we denote the following subcategories of~\(\mathcal{A}\):
\begin{itemize}
\item \(\Summ(\mathcal{C})\) the subcategory of \(\mathcal{A}\) consisting of objects isomorphic
to~direct summands of elements from \(\mathcal{C}\),
\item \(\Coprod(\mathcal{C})\) the subcategory of \(\mathcal{A}\) consisting of objects isomorphic
to~coproducts of elements from \(\mathcal{C}\),
\item \(\coprodf(\mathcal{C})\) the subcategory of \(\mathcal{A}\) consisting of objects isomorphic
to~finite coproducts of elements from \(\mathcal{C}\),
\item \(\Add(\mathcal{C})\) the subcategory \(\Summ\big(\Coprod(\mathcal{C})\big)\),
\item \(\addf(\mathcal{C})\) the subcategory \(\Summ\big(\coprodf(\mathcal{C})\big)\),
\item  \(\mathcal{C}^{\perp}\) the subcategory \(\{A\in \mathcal{A}:\forall C\in\mathcal{C},\Hom_{\mathcal{A}}(C,A)=0\}\),
\item \({}^{\perp}\mathcal{C}\) the subcategory \(\{A\in \mathcal{A}:\forall C\in\mathcal{C},\Hom_{\mathcal{A}}(A,C)=0\}\).
\end{itemize}

If \(\mathcal{C}=\{C\}\) for a single object \(C\in\mathcal{C}\), we often omit the brackets and write \(\Summ(C)\) etc.\ instead.

\subsubsection{Triangulated categories}
\label{NotationTriangulatedCategories}
The notational conventions about triangulated categories here are primarily based on \cite{Neeman01} and \cite[Subsection~2.1]{BiswasChenRahulParkerZheng24}.

By default, we will denote the shift functor on a~triangulated category \(\mathcal{T}\) as \(\Sigma\).
The triangulated subcategory of \(\mathcal{T}\) consisting of its compact elements is denoted as~\(\mathcal{T}^c\).
And for subcategories \(\mathcal{A},\mathcal{B}\subseteq\mathcal{T}\) we denote
\(\mathcal{A}*\mathcal{B}\) the subcategory of \(\mathcal{T}\) with objects
\[
\{X\in\mathcal{T}:\exists A\in\mathcal{A},\exists B\in\mathcal{B},\exists \text{ a triangle } A\rightarrow X\rightarrow B\rightarrow \Sigma A \text{ in }\mathcal{T}\}
.\]
A subcategory \(\mathcal{S}\subseteq\mathcal{T}\) is \textit{extension-closed} if \(\mathcal{S}=\mathcal{S}*\mathcal{S}\).

For every object $X$ of a triangulated category \(\mathcal{T}\) and \(N,M\in\Z\) we define a set \(X[N,M]:=\{\Sigma^{-i}X:i\in\Z,N\leq i\leq M\}\).
We also define:
\begin{itemize}
\item
\(X[N,\infty]:=\{\Sigma^{-i}X:i\in\Z,N\leq i\}\), 
\item 
\(X[-\infty,M]:=\{\Sigma^{-i}X:i\in\Z, i\leq M\}\),
\item
\(X[-\infty,\infty]:=\{\Sigma^{-i}X:i\in\Z\}\).
\end{itemize}
The smallest subcategory of \(\mathcal{T}\) closed under extensions and direct summands containing \(X[N,M]\) will be denoted \(\langle X\rangle^{[N,M]}\).

More generally, for a class \(\mathcal{C}\subseteq\Ob(\mathcal{T})\) we define \(\mathcal{C}[N,M]:=\bigcup_{C\in\mathcal{C}}C[N,M]\). We treat the cases \(\mathcal{C}[N,\infty]\), \(\mathcal{C}[-\infty,M]\), \(\mathcal{C}[-\infty,\infty]\), and \(\langle\mathcal{C}\rangle^{[N,M]}\) analogously.
In particular, the notation \(\langle \mathcal{C}\rangle:=\langle \mathcal{C}\rangle^{[-\infty,\infty]}\) stands for \textit{the thick subcategory} of \(\mathcal{T}\) generated by \(\mathcal{C}\).

Assuming that the category \(\mathcal{T}\) has countable coproducts, we define \textit{a homotopy colimit} \(\hoco E_*\) of a directed diagram \(E_*:= E_1\xrightarrow{f_{1,2}} E_2 \xrightarrow{f_{2,3}} \cdots\) in \(\mathcal{T}\) by the~distinguished triangle 
\[
\coprod_{n\in\N} E_n \xrightarrow{\ \varphi\ } \coprod_{n\in\N} E_n \xrightarrow{\quad} \hoco E_* \xrightarrow{\quad} \Sigma \coprod_{n\in\N} E_n
\]
where the first map is given by the infinite matrix
\[
\varphi:=
\left[
\begin{array}{cccc}
\id_{E_1} & 0           & 0           & \cdots \\
-f_{1,2}  & \id_{E_2}   & 0           & \cdots \\
0         & -f_{2,3}    & \id_{E_3}   & \cdots \\[-0.4ex]
\vdots    & \vdots      & \vdots      & \ddots
\end{array}
\right]
.\]

\subsubsection{Homotopy categories and derived categories}
Let \(\mathcal{A}\) be an exact category. We use the following notation for relevant related categories:
\begin{itemize}
    \item \(\HomotopyCategory(\mathcal{A})\) the homotopy category of \(\mathcal{A}\),
    \item \(\derived(\mathcal{A})\) the derived category of \(\mathcal{A}\),
    \item \(\derived^b(\mathcal{A})\) the bounded derived category of \(\mathcal{A}\).
\end{itemize}

\subsubsection{Rings and modules}
By ring, we always mean associative, unitary ring whose operation of multiplication generally does not have to be commutative.
By module, we mean a right module over a given ring unless stated otherwise.
A $K$-algebra over a~field $K$ is a~ring containing $K$ as a subring of its center.

For a ring $R$, we use the following notation for relevant related categories:
\begin{itemize}
    \item \(\ModR\) the category of $R$-modules,
    \item \(\modf\dashmodule R\) the category of finitely generated $R$-modules,
    \item \(\ind\dashmodule R\) the category of finitely generated indecomposable $R$-modules,
    \item \(\derived(R)\) abbreviation for \(\derived(\ModR)\),
    \item \(\derived^b(R)\) abbreviation for \(\derived^b(\ModR)\). 
\end{itemize}

Furthermore, for $R$, $S$ rings and \({}_RM_S\) an $R$-$S$-bimodule we denote
\[\blank\otimes_R^{\mathbf{L}}M:\derived(R)\rightarrow\derived(S) \text{ the left derived functor of }\blank\otimes_RM:\ModR\rightarrow\Mod\dashmodule S\]
and
\[\mathbf{R}\Hom_{S}(M,\blank):\derived(S)\rightarrow\derived(R)\] the right derived functor of \(\Hom_{S}(M,\blank):\Mod\dashmodule S\rightarrow\ModR\).

\subsection{Metrics on triangulated categories}

As explained in \cite[Heuristic~9]{Neeman20}, we are aiming to assign lengths (with values in rational numbers) to morphisms in~a~triangulated theory \(\mathcal{T}\) in such a way that the length is invariant under homotopy cartesian squares, i.e.\ if
 \[
\begin{tikzcd}
    A \arrow{d}{g} \arrow{r}{f} & B \arrow{d}{g'} \\
    C \arrow{r}{f'}                & D               
\end{tikzcd}
\]
is a homotopy cartesian square,
then the morphisms $f$ and $f'$ have equal length.
We recall that the commutative square as above is called \textit{homotopy cartesian} if there exists a distinguished triangle 
\[
A \xrightarrow{\begin{bsmallmatrix}f \\ -g\end{bsmallmatrix}} B\oplus C\xrightarrow{\begin{bsmallmatrix}g' & f'\end{bsmallmatrix}} D \xrightarrow{\quad} \Sigma A
.\]
Consequently, we may measure a length of a morphism by looking at its cone, because for every distinguished triangle \(A\xrightarrow{f} B\rightarrow C\rightarrow \Sigma A\) there exists a homotopy cartesian square
 \[
\begin{tikzcd}
    A \arrow{d} \arrow{r}{f} & B \arrow{d} \\
    0 \arrow{r}                & C.               
\end{tikzcd}
\]
In other words, the length of all morphisms is determined by the lengths of~the~zero morphisms \(0\rightarrow C\) for all \(C\in\mathcal{T}\).
This motivates the following definition of a metric on a triangulated category as a decreasing sequence of neighbourhoods of~zero.

This definition of a metric comes from {\cite[Definition 1.2]{Neeman18}}, while the definition of~a~good metric is from {\cite[Definition 10]{Neeman20}}. 

\begin{definition}
\label{DefinitionMetric}
Let $\mathcal{T}$ be a triangulated category.
Let \(\{B_n\}_{n\in\N}\) be a non-increasing chain \(B_1\supseteq B_2 \supseteq B_3\supseteq\ldots\) of subcategories of \(\mathcal{T}\).
We say that \(\{B_n\}_{n\in\N}\) is \textit{a~metric} if \(B_n*B_n=B_n\) and \(0\in B_n\) for all \(n\in\N\).

Furthermore,
we say that a metric \(\{B_n\}_{n\in\N}\) is \textit{a good metric}
if it satisfies \textit{the~rapid decrease condition}, i.e. \(\Sigma^{-1} B_{n+1}\cup B_{n+1}\cup \Sigma B_{n+1}\subseteq B_n\)
for all \(n\in\N\).
\end{definition}

\begin{definition}[{\cite[Definition 1.6]{Neeman18}}]
Let $\mathcal{T}$ be a triangulated category with a~metric \(\{B_n\}_{n\in\N}\).
Let \(E_*:= E_1\xrightarrow{f_{1,2}} E_2 \xrightarrow{f_{2,3}} \cdots\) be a directed diagram in \(\mathcal{T}\).
For \(m'> m\in\N\), we denote the composition \(f_{m,m'}:=f_{m',m'-1}\circ\cdots\circ f_{m+1,m+2}\circ f_{m,m+1}\) and \(f_{m,m}:=\id_{E_m}\).

We say that \(E_*\) is \textit{a Cauchy sequence} (with respect to the metric \(\{B_n\}_{n\in\N}\)) if for all \(i\in\N\) there is an~index \(M\in\N\) such that for every \(m'\geq m\geq M\) the triangle \(E_m\xrightarrow{f_{m,m'}} E_{m'}\rightarrow D_{m,m'}\rightarrow\Sigma E_m\) satisfies \(D_{m,m'} \in B_i\).

We say that \(E_*\) is \textit{a good Cauchy sequence} 
(with respect to the metric \(\{B_n\}_{n\in\N}\))
if for all \(i\in\N\) and \(j\in\Z\) there is an index \(M\in\N\) such that for every \(m'\geq m\geq M\) the~triangle \(E_m\xrightarrow{f_{m,m'}} E_{m'}\rightarrow D_{m,m'}\rightarrow\Sigma E_m\) satisfies \(\Sigma^j D_{m,m'} \in B_i\).
\end{definition}

\begin{remark}
The existence of a metric \(\{B_n\}_{n\in\N}\) on \(\mathcal{T}\) does not, strictly speaking, assign to a morphism \(f:a\rightarrow b\) a precise number to be the length of $f$. Instead, we are merely interpreting \(\cone(f)\in B_n\) for some fixed \(n\in\N\) as~the~length of $f$ being less or equal to \(\frac{1}{n}\).
Since the cone of every isomorphism, the zero object, is contained in the intersection \(\bigcap_{n\in\N}B_n\), all isomorphisms are of length \(\leq\frac{1}{n}\) for all \(n\in\N\). In this sense, we can say that isomorphisms are of zero length.
And Cauchy sequences are sequences of composable morphisms whose lenghts are eventually getting smaller and smaller
(see \cite[Definition 2]{Neeman20}).

The original definition of a good metric \(\{B_n\}_{n\in\N}\) in {\cite[Definition~10]{Neeman20}} contains an~additional assumption that \(B_1=\mathcal{T}\). This was done to ensure that every morphism has a length, concretely length \(\leq1\).
Since we are interested mainly in~Cauchy sequences \(E_*:= E_1\xrightarrow{f_{1,2}} E_2 \xrightarrow{f_{2,3}} \cdots\) anyway, where we impose the~length of~\(E_m\xrightarrow{f_{m,m'}}E_{m'}\) to be small for indices \(m'\geq m\geq M\) for some fixed $M\in\N$ onwards, the requirement of every morphism having lengths \(\leq1\) is not necessary. We can simply skip any finite initial segment of a Cauchy sequence.

Therefore, in this text, we will stick to the variant of the definition of a good metric from \cite[Definition 2.8]{BiswasChenRahulParkerZheng24} where we allow morphisms without any length by~omitting the condition \(B_1=\mathcal{T}\).
This will simplify the notation and formulations of~certain statements and their proofs.
\end{remark}

Similarly to the case of metrics on metric spaces, we can define an equivalence relation on metrics on triangulated categories. 

\begin{definition}[{\cite[Definition 1.2]{Neeman18}}]
    Let $\mathcal{T}$ be a triangulated category and two metrics \(\{B_n\}_{n\in\N}\) and~\(\{C_n\}_{n\in\N}\).
    We say that \(\{C_n\}_{n\in\N}\) is \textit{finer} than (or \textit{a refinement} of) \(\{B_n\}_{n\in\N}\) if for every  \(n\in\N\) there is \(m\in\N\) such that \(C_m\subseteq B_n\).
    
    In~this~case, we also say that \(\{B_n\}_{n\in\N}\) is \textit{coarser} than \(\{C_n\}_{n\in\N}\), and we denote this by~\(\{C_n\}_{n\in\N}\leq\{B_n\}_{n\in\N}\).

    We say that \(\{B_n\}_{n\in\N}\) and \(\{C_n\}_{n\in\N}\) are \textit{equivalent} if
    \[
    \{C_n\}_{n\in\N}\leq\{B_n\}_{n\in\N} \text{\ \ \ and\ \ \ }\{B_n\}_{n\in\N}\leq\{C_n\}_{n\in\N}.
    \]
\end{definition}

\begin{remark}[{\cite[Remark~1.7]{Neeman18}}]
    Equivalent metrics yield the same (good) Cauchy sequences.
\end{remark}

\begin{remark}
\label{EquivalenceDoesntPreserveGoodness}
Equivalence of metrics does not preserve goodness.
Let $\mathbf{D}(K)$ be the~category of graded vector spaces over a field $K$.
Consider a good metric \(\{C_n\}_{n\in\N}\) given by the cohomology functor \mbox{\(H:\mathbf{D}(K)\rightarrow\Mod\dashmodule K\)} via the formula
\[
C_n=\{T\in\mathbf{D}(K):\forall i>-n,H^i(T)=0\}
.\]
This is a particular case of a metric of ``type i)'' from \cite[Example~12]{Neeman20}.

The metric \(\{C_n\}_{n\in\N}\) is indeed good because for each \(n\in\N\) and \(T\in C_{n+1}\), the~property \(H^{i}(T)=0\) for \(i>-n-1\) implies \(H^j\left(\Sigma^{-1}T\right)=0\) whenever \(j>-n\), thus justifying \(\Sigma C_{n+1}\cup C_{n+1}\cup\Sigma^{-1}C_{n+1}\subseteq C_n\).

Define a metric \(\{B_n\}_{n\in\N}\) by the formula \(B_n:=C_{\lceil \frac{n}{2} \rceil}\) for all \(n\in\N\).
This metric is not good since \(\Sigma^{n}K\in B_{2n}=C_{n}\) for all \(n\in\N\) while \(\Sigma^{n-1}K\notin B_{2n-1}=C_{n}\).

We observe that \(\{B_n\}_{n\in\N}\) and \(\{C_n\}_{n\in\N}\) are equivalent, even though \(\{C_n\}_{n\in\N}\) is good, while \(\{B_n\}_{n\in\N}\) is not.
\end{remark}

\subsection{Completions of triangulated categories}

The construction of a completion of a triangulated category \(\mathcal{T}\) resembles the one of a completion of a metric space. We are formally adjoining ``limit points'' of Cauchy sequences - in our case directed colimits of images of Cauchy sequences under the Yoneda embedding to~obtain the~pre-completion \(\mathfrak{L}(\mathcal{T})\) (see Definition~\ref{DefinitionCompletion}).

However, unlike the completion of a metric space, the pre-completion \(\mathfrak{L}(\mathcal{T})\) is not yet what we wanted because it is not known, whether it is, in general, a triangulated category. To~get hold of~the~true completion \(\mathfrak{S}(\mathcal{T})=\mathfrak{C}(\mathcal{T})\cap\mathfrak{L}(\mathcal{T})\), we must restrict ourselves
to those elements of \(\mathfrak{L}(\mathcal{T})\) which are compactly supported. This ensures that the~completion \(\mathfrak{S}(\mathcal{T})\) is triangulated with a shift functor compatible with the~original shift on \(\mathcal{T}\).

The reasoning behind the introduction of compactly supported elements is explained in Remark~\ref{SendsSmallMorphismsToIsomorphisms} and Remark~\ref{Explanation0}.

A more precise definition follows.

\begin{definition}[{\cite[Definition 1.10]{Neeman18}}]
\label{DefinitionCompletion}
Let $\mathcal{T}$ be a triangulated category with a~metric \(\{B_n\}_{n\in\N}\).
We define the following subcategories of \(\Mod\dashmodule\mathcal{T}\) (all of~them with respect to~\(\{B_n\}_{n\in\N}\)):
\begin{enumerate}[label=(\roman*)]
\item \textit{the compactly supported elements} as \(\mathfrak{C}(\mathcal{T}):=\bigcap_{j\in\Z}\bigcup_{i\in\N}Y(\Sigma^j B_i)^{\perp}\),
\item \textit{the weakly compactly supported elements} as \(\mathfrak{WC}(\mathcal{T}):=\bigcup_{i\in\N}Y(B_i)^{\perp}\),
\item \textit{the pre-completion} of \(\mathcal{T}\) as
\[\ \ \ \ \ \ \mathfrak{L}(\mathcal{T}):=\left\{F\in\Mod\dashmodule\mathcal{T}:\exists \text{ a good Cauchy sequence \(E_*\) in \(\mathcal{T}\)}, F\simeq\varinjlim Y(E_*)\right\},\]
\item and \textit{the completion} of \(\mathcal{T}\) as \(\mathfrak{S}(\mathcal{T}):=\mathfrak{C}(\mathcal{T})\cap\mathfrak{L}(\mathcal{T})\).
\end{enumerate}

We also define an autoequivalence \(\Sigma:\Mod\dashmodule\mathcal{T}\rightarrow\Mod\dashmodule\mathcal{T}\) by the formula \([\Sigma F](T):= F(\Sigma^{-1}T)\) for~all \(F\in\Mod\dashmodule\mathcal{T}\) and \(T\in\mathcal{T}\).

A directed diagram
\[
\begin{tikzcd}
a_1 \arrow[d] \arrow[r, "f_1"] & b_1 \arrow[d] \arrow[r, "g_1"] & c_1 \arrow[d] \arrow[r, "h_1"] & \Sigma a_1 \arrow[d] \\
a_2 \arrow[d] \arrow[r, "f_2"] & b_2 \arrow[d] \arrow[r, "g_2"] & c_2 \arrow[d] \arrow[r, "h_2"] & \Sigma a_2 \arrow[d] \\
\vdots & \vdots & \vdots & \vdots
\end{tikzcd}
\]
of distinguished triangles in \(\mathcal{T}\) is said to be \textit{a good Cauchy sequence of distinguished triangles} provided that each of the sequences \(a_*\), \(b_*\), \(c_*\), and \(\Sigma a_*\) is good Cauchy. 

A sequence \(A\xrightarrow{f}B\xrightarrow{g}C\xrightarrow{h}\Sigma A\) in \(\mathfrak{S}(\mathcal{T})\) is then called \textit{a distinguished triangle} if it is isomorphic to the directed colimit of a good Cauchy sequence of distinguished triangles \(a_*\xrightarrow{f_*}b_*\xrightarrow{g_*}c_*\xrightarrow{h_*}\Sigma a_*\)  in~\(\mathcal{T}\) embedded into \(\Mod\dashmodule\mathcal{T}\) via Yoneda's embedding.
\end{definition}

In case that the category \(\mathcal{T}\) is not skeletally small, natural transformations between two functors in \(\Mod\dashmodule\mathcal{T}\) does not have to form a set. However, both of the~categories \(\mathfrak{L}(\mathcal{T})\) and \(\mathfrak{S}(\mathcal{T})\), which we are mainly interested in, consist of countably presented functors only. This makes them genuine categories with small Hom-sets regardless of any set-theoretical issues with \(\mathcal{T}\).

Since we will be often working with directed colimits \(\varinjlim Y(E_*)\) of~(good) Cauchy sequences \(E_*\) from \(\mathcal{T}\) calculated inside the module category \(\Mod\dashmodule\mathcal{T}\), we introduce a~special notation and terminology for such a colimit.  

\begin{definition}[{\cite[Definition, 3.4]{CummingsGratz24}}]
Let \(\mathcal{T}\) be a triangulated category with a~metric. Let \(E_*\) be a~Cauchy sequence in \(\mathcal{T}\). We call \(\varinjlim Y(E_*)\in\mathfrak{L}(\mathcal{T})\) \textit{the module colimit} of \(E_*\) and we denote it as \(\moco E_*\). 
\end{definition}

\begin{remark}[{\cite[Remark~1.11]{Neeman18}}]
Equivalent metrics yield the same (weakly) compactly supported elements and the~same pre-com\-ple\-tion. Consequently, they yield the same completion. 
\end{remark}

\begin{remark}
It holds that \(\Sigma\big(\mathfrak{S}(\mathcal{T})\big)=\mathfrak{S}(\mathcal{T})\). In other words, $\Sigma$ restricts to~an~autoequivalence of~\(\mathfrak{S}(\mathcal{T})\). This is explained in more detail in \cite[Chapter~2]{Neeman18}.
\end{remark}

We finally get to Neeman's result that the completion (as above) of triangulated category is triangulated itself.

\begin{theorem}[{\cite[Theorem 2.11]{Neeman18}}]
\label{NeemansCompletionTheorem}
Let $\mathcal{T}$ be a triangulated category with shift \(\Sigma\) and a metric \(\{B_n\}_{n\in\N}\).
Then the category \(\mathfrak{S}(\mathcal{T})\) together with its autoequivalence \(\Sigma\) and its distinguished triangles from Definition~\ref{DefinitionCompletion} is triangulated. 
\end{theorem}

In order to illustrate the importance of compactly supported elements, we state one of the auxiliary lemmas needed in the proof of Theorem~\ref{NeemansCompletionTheorem}. 
Not only is it an essential tool to show that the octahedral axiom holds in \(\mathfrak{S}(\mathcal{T})\) but it is an~interesting statement by itself. It will, in fact, be used again in Section~\ref{sec:CertainCompletions}.

\begin{lemma}[{\cite[Lemma 2.8]{Neeman18}}]
\label{BasicFactorisationProperty}
Let $\mathcal{T}$ be a triangulated category with a metric.
Let \(F\in\mathfrak{C}(\mathcal{T})\) be a cohomological functor on \(\mathcal{T}\) and let \(E=\moco E_*\in\mathfrak{L}(\mathcal{T})\) for~some good Cauchy sequence \(E_*\).
Then there exists \(N\in\N\) such that for all \(n\geq N\) every map \(Y(E_n)\rightarrow F\) in \(\Mod\dashmodule\mathcal{T}\) factors uniquely as \(Y(E_n)\xrightarrow{\varphi} E\rightarrow F\) where \(\varphi\) is the~canonical map into the colimit.
\end{lemma}

Lemma~\ref{BasicFactorisationProperty} concerns the unique factorisation of a certain morphism to~a~com\-pact\-ly supported cohomological functor $F$ through a morphism from an element $E$ of~the~pre-completion.
As the completion \(\mathfrak{S}(\mathcal{T})=\mathfrak{C}(\mathcal{T})\cap\mathfrak{L}(\mathcal{T})\) by~definition consists of~the~elements of~the~pre-completion \(\mathfrak{L}(\mathcal{T})\) which are compactly supported, it applies to any pair of elements $E$ and $F$ from the completion.

The following two remarks are based on \cite[Definition~4]{Neeman20} and the proof of aforementioned Lemma~\ref{BasicFactorisationProperty} in {\cite[Lemma~2.8]{Neeman18}}.

\begin{remark}
\label{SendsSmallMorphismsToIsomorphisms}
Let $\mathcal{T}$ be a triangulated category with a metric \(\{B_n\}_{n\in\N}\).
Let \(F\in\mathfrak{C}(\mathcal{T})\). If $F$ is a cohomological functor on \(\mathcal{T}\),
then it sends sufficiently small morphisms to isomorphisms.

By the definition of compactly supported elements, we can find \(n\in\N\) such that \(F\in Y(\Sigma^{-1} B_n)^{\perp}\cap Y(B_n)^{\perp}\).
Let \(f:C\rightarrow D\) in \(\mathcal{T}\) be of length at most~\(\frac{1}{n}\),
i.e.\ there exists a triangle \(C\xrightarrow{f} D\rightarrow E\rightarrow\Sigma C\) with \(E\in B_n\).
Applying the cohomological functor $F$ to~this triangle yields a long exact sequence
\[
\cdots \rightarrow F(E) \rightarrow F(D) \xrightarrow{f} F(C) \xrightarrow{} F(\Sigma^{-1}E) \rightarrow \cdots 
\]
of abelian groups.
By Yoneda's lemma it holds that
\[
F(E)\simeq \Hom\big(Y(E),F\big)=0 \text{\ \ \ \ and\ \ \ \ } F(\Sigma^{-1}E)\simeq \Hom\big(Y(\Sigma^{-1}E),F\big)=0
,\]
so $F(f)$ must be an isomorphism thanks to the exactness of our sequence.

This fact is a key step in the proof of Lemma~\ref{BasicFactorisationProperty}, and consequently of the~octahedral axiom for \(\mathfrak{S}(\mathcal{T})\) as well.
\end{remark}

\begin{remark}
\label{Explanation0}
Note that a functor \(F\in\mathfrak{L}(\mathcal{T})\) is automatically cohomological as a~directed colimit of representable functors, which are known to be cohomological.
Therefore, any functor in the completion \(\mathfrak{S}(\mathcal{T})=\mathfrak{C}(\mathcal{T})\cap\mathfrak{L}(\mathcal{T})\) is both compactly supported and cohomological, so it satisfies the property of ``sending small morphisms to~isomorphisms'' from Remark~\ref{SendsSmallMorphismsToIsomorphisms}.
\end{remark}

\subsection{Good extensions}
\label{sub:GoodExtensions}

A good extension is an indispensable tool to explicitly calculate completions of triangulated categories without referring to the functor category. If a triangulated category embeds into another triangulated category with coproducts in such a way that the embedding is a good extension, the completion is then triangle equivalent to a suitable subcategory of the extension.

The following definition is comprised of Notation~3.1, Definition~3.5, and Definition~3.8 from \cite{Neeman18}.
For the definition of a homotopy colimit, we refer to Subsection~\ref{NotationTriangulatedCategories}.

\begin{definition}
\label{DefinitionCompletionGoodExtension}
Let \(\mathcal{S},\mathcal{T}\) be triangulated categories and \(F:\mathcal{S}\rightarrow\mathcal{T}\) a fully faithful, triangulated functor.
We define \textit{the restricted Yoneda functor} \mbox{\(\mathcal{Y}:\mathcal{T}\rightarrow\Mod\dashmodule\mathcal{S}\)} by setting \(T\mapsto\Hom\big(F(\blank),T\big)\).

Suppose that \(\mathcal{S}\) has a metric \(\{B_n\}_{n\in\N}\).
We say that $F$ is \textit{a good extension} (with respect to \(\{B_n\}_{n\in\N}\)) if \(\mathcal{T}\) has countable coproducts
and for every good Cauchy sequence \(E_*\)
the canonical map \(\moco E_*\rightarrow\mathcal{Y}\big( \hoco F(E_*)\big)\) is an isomorphism.

In this setting, we define the following subcategories of \(\mathcal{T}\):
\begin{itemize}[label={},leftmargin=0.5cm]
\item \(\mathfrak{C}'(\mathcal{S}):=\mathcal{Y}^{-1}\big(\mathfrak{C}(\mathcal{S})\big)\),
\item \(\mathfrak{WC}'(\mathcal{S}):=\mathcal{Y}^{-1}\big(\mathfrak{WC}(\mathcal{S})\big)\),
\item \(\mathfrak{L}'(\mathcal{S}):=\left\{T\in\mathcal{T}:\exists \text{ a good Cauchy sequence \(E_*\) in \(\mathcal{S}\) }, T\simeq\hoco F(E_*)\right\}\),
\item \(\mathfrak{S}'(\mathcal{S}):=\mathfrak{C}'(\mathcal{S})\cap\mathfrak{L}'(\mathcal{S})\).
\end{itemize}
\end{definition}

\begin{remark}[{\cite[Observation 3.2]{Neeman18}}]
We can describe the categories \(\mathfrak{C}'(S)\) and \(\mathfrak{WC}'(\mathcal{S})\) explicitly as \(\mathfrak{C}'(S)=\bigcap_{j\in\Z}\bigcup_{i\in\N}F(\Sigma^j B_i)^{\perp}\) and
\(\mathfrak{WC}'(\mathcal{S})=\bigcup_{i\in\N}F(B_i)^{\perp}\).
\end{remark}

We can calculate completions inside good extensions.

\begin{theorem}[{\cite[Theorem 3.15]{Neeman18}}]
\label{TheOnlyComputationalTool}
Let \(\mathcal{S}\) be a triangulated category with a~metric, and \(\mathcal{S}\xhookrightarrow{}\mathcal{T}\) a good extension.
Then \(\mathcal{Y}\restriction_{\mathfrak{S}'(\mathcal{S})}:\mathfrak{S}'(\mathcal{S})\rightarrow\mathfrak{S}(\mathcal{S}) \) is a well-defined triangulated equivalence.
\end{theorem}

We mention a technical lemma from the proof of the theorem.

\begin{lemma}[{\cite[Lemma 3.13]{Neeman18}}]
\label{ExtensionFakeYonedaLemma}
Let \(\mathcal{S}\) be a triangulated category with a metric, and
\(\mathcal{S}\xhookrightarrow{}\mathcal{T}\) a good extension.
Let \(E\in\mathfrak{L}'(\mathcal{S})\) and \(X\in\mathfrak{C}'(\mathcal{S})\). Then the map
\[
\mathcal{Y}:\Hom_{\mathcal{T}}(E,X)\rightarrow\Hom_{\Mod\dashmodule\mathcal{S}}\big(\mathcal{Y}(E),\mathcal{Y}(X)\big)
\]
is an isomorphism.
\end{lemma}

The isomorphism from Lemma~\ref{ExtensionFakeYonedaLemma} enables us to translate the unique factorisation property of Lemma~\ref{BasicFactorisationProperty} into the setting of the good extension. This will be useful again in Section~\ref{sec:CertainCompletions}.

\begin{lemma}
\label{BasicExtensionFactorisationProperty}
Let \(\mathcal{S}\) be a triangulated category with a metric, and \(F:\mathcal{S}\rightarrow{}\mathcal{T}\) a good extension.
Let \(X\in\mathfrak{C}'(\mathcal{S})\) and let \(E=\hoco F(E_*)\in\mathfrak{L}'(\mathcal{T})\) for some good Cauchy sequence \(E_*\).
Then there exists \(N\in\N\) such that for all \(n\geq N\) every map \(F(E_n)\rightarrow X\) in \(\mathcal{T}\) factors uniquely as \(F(E_n)\rightarrow E\rightarrow X\).
\end{lemma}

\begin{proof}
Since \(\mathcal{Y}(X)\) is a cohomological functor, Lemma~\ref{BasicFactorisationProperty} gives us an index \(N\in\N\) and~for~each \(n\geq N\) a~unique factorisation of \(\mathcal{Y}\circ F(E_n)\simeq Y(E_n)\rightarrow \mathcal{Y}(X)\) as
\[
\mathcal{Y}\circ F(E_n)\simeq Y(E_n)\rightarrow \moco E_*\simeq\mathcal{Y}(E)\rightarrow \mathcal{Y}(X)
.\]
We then lift the factorisation to \(\mathcal{T}\) using Lemma~\ref{ExtensionFakeYonedaLemma}.
\end{proof}

We finish the section by mentioning the prototypical example of a good extension, i.e.\ the embedding of compact elements \(\mathcal{T}^c\xhookrightarrow{}\mathcal{T}\).

\begin{theorem}[{\cite[Example~3.6]{Neeman18}}]
\label{PrototypicalGoodExtension}
Let \(\mathcal{T}\) be a triangulated category with arbitrary coproducts. Then the inclusion \(\mathcal{T}^c\xhookrightarrow{}\mathcal{T}\) is a good extension with respect to~all metrics on \(\mathcal{T}^c\).
\end{theorem}

\begin{proof}
Let \(E_*:= E_1\xrightarrow{} E_2 \xrightarrow{} \cdots\) be a directed diagram in \(\mathcal{T}^c\).
Then we can use \cite[Lemma~2.8]{Neeman96} to show that the canonical map \(\varinjlim Y(E_*)\rightarrow\mathcal{Y}( \hoco E_*)\) is an~isomorphism regardless~of whether the sequence \(E_*\) is good Cauchy or not. Hence the inclusion is a good extension for every possible metric on \(\mathcal{T}^c\).
\end{proof}

\section{Completions of certain derived categories}
\label{sec:CertainCompletions}

This section is dedicated to calculating all completions of the bounded derived category \(\derived^b(\modf\dashmodule A)\) of finitely generated modules over a hereditary, finite dimensional algebra of finite representation type. Such a category is equal to the category of \(\derived^c(A)\) of the compact elements (equivalently perfect complexes) of the derived category \(\derived(A)\). See for example \cite[Proposition~6.4]{BokstedtNeeman93}. We are using the fact that the subcategory of perfect complexes is equal to  \(\derived^b(\modf\dashmodule A)\) for $A$ of finite global dimension.

Throughout the section, 
we will be always implicitly using Theorem~\ref{TheOnlyComputationalTool} as our computation tool because embedding \(\derived^b(\modf\dashmodule A)\xhookrightarrow{}\derived(A)\) is a good extension for~every metric (see Theorem~\ref{PrototypicalGoodExtension}).
In other words, all the completions mentioned below will be of the form \(\mathfrak{S}'(\mathcal{T})\) in the sense of the notation from Definition~\ref{DefinitionCompletionGoodExtension}.

Firstly, we show that in the case of a hereditary category \(\mathcal{A}\) the completion of~\(\derived(\mathcal{A})^c\) is always bounded in cohomology.
Recall that any object of \(\derived(\mathcal{A})\) for~\(\mathcal{A}\) hereditary is quasi-isomorphic to its cohomology (see for example \cite[Proposition~4.4.15]{Krause21}).

\begin{proposition}
\label{BoundedHereditaryCompletion}
Let \(\mathcal{A}\) be a hereditary Abelian category with coproducts. Consider the~subcategory of perfect complexes \(\derived(\mathcal{A})^c\) of the derived category \(\derived(\mathcal{A})\) of~\(\mathcal{A}\). Let \(\mathcal{S}\) be a triangulated subcategory of~\(\derived(\mathcal{A})^c\).
Then for every metric on \(\mathcal{S}\), we have \(\mathfrak{S}'(\mathcal{S})\subseteq\derived^b(\mathcal{A})\).
\end{proposition}

\begin{proof}
Let \(E\in\mathfrak{S}'(\mathcal{S})\) where \(E=\hoco E_*\) for some Cauchy sequence \(E_*\).
By~Lemma~\ref{BasicExtensionFactorisationProperty}, for every \(X\in\mathfrak{C}'(\mathcal{S})\) there exists \(N\in\N\) such that every map \(E_N\rightarrow X\) in \(\derived(\mathcal{A})\) factors uniquely as \(E_N\rightarrow E\rightarrow X\).
We find such an \(N\in\N\) for~the~choice \(X:=E\in\mathfrak{C}'(\mathcal{S})\).

Suppose that \(E_N\) has cohomology concentrated in degrees \(-n,\ldots,n\) for some \(n\in\N\).
Denote \(Q:=\coprod_{i=-n}^{n+1}\Sigma^{-i}H^i(E)\in\derived^b(\mathcal{A})\).
Since \(\mathcal{A}\) is hereditary, $Q$ is a~direct summand of $E$, so we can consider the~idempotent map \(e:E\rightarrow E\) having $Q$ as its image.

Because there are no morphisms \(E_N\rightarrow \Sigma^iZ\) for every \(Z\in\mathcal{A}\) and every \(i\in\Z\), \(i\notin\{-n,\ldots,n+1\}\),
the canonical map \(E_N\rightarrow E\) factors through $Q$.
But then the~compositions \(E_N\rightarrow E \xrightarrow{\id_E} E\) and \(E_N\rightarrow E \xrightarrow{e} E\) are equal. By the uniqueness of such a factorisation, we get \(e=\id_E\) and \(E\simeq Q\in\derived^b(\mathcal{A})\).
\end{proof}

For the rest of the section, we narrow our attention to finite dimensional algebras.

\begin{lemma}
\label{FinitenessOfIndexingSet}
Let \(A\) be a hereditary finite dimensional $K$-algebra and \(\mathcal{S}\) a triangulated subcategory of~\(\derived^b(\modf\dashmodule A)\).
Let \(i\in\Z\).
Let \(E=\hoco E_*\in\mathfrak{S}'(\mathcal{S})\) for~some good Cauchy sequence \(E_*\) with respect to a fixed metric on~\(\mathcal{S}\).
If \(H^i(E)\) contains \(\bigoplus_{j\in J}M\) as a direct summand for some \(0\neq M\in \modf\dashmodule A\) and some indexing set $J$, then $J$ is finite.
\end{lemma}

\begin{proof}
Since $\Sigma^{-i}M$ is a direct summand of $E$, it has to be compactly supported (see~Definition~\ref{DefinitionCompletion}). 
By Lemma~\ref{BasicExtensionFactorisationProperty}, there is \(N\in\N\) such that the map
\[
\Hom_{\derived(A)}\left(\varphi,\Sigma^{-i}M\right):\Hom_{\derived(A)}\left(E,\Sigma^{-i}M\right)\rightarrow\Hom_{\derived(A)}\left(E_N,\Sigma^{-i}M\right)
\]
is a~$K$\=/linear isomorphism, where \(\varphi:E_N\rightarrow E\) is the canonical map.

Because \(E_N,\Sigma^{-i}M\in\derived^b(\modf\dashmodule A)\), the dimension of \(\Hom_{\derived(A)}(E_N,\Sigma^{-i}M)\) is finite.
As \(\bigoplus_{j\in J}\Sigma^{-i}M\) is a direct summand of $E$, we know
\[
\begin{split}
0&\neq\prod_{j\in J}\Hom_{\derived(A)}\left(\Sigma^{-i}M,\Sigma^{-i}M\right)\simeq
\Hom_{\derived(A)}\left(\bigoplus_{j\in J}\Sigma^{-i}M,\Sigma^{-i}M\right) \\
&\subseteq\Hom_{\derived(A)}\left(E,\Sigma^{-i}M\right)\simeq\Hom_{\derived(A)}\left(E_N,\Sigma^{-i}M\right).
\end{split}
\]
If $J$ was infinite, the dimension of \(\Hom_{\derived(A)}(E_N,\Sigma^{-i}M)\) would be infinite, but that is impossible. Hence, the set $J$ is finite.
\end{proof}

We recall the following result of Auslander about modules over artinian rings of~finite representation type. It applies, in particular, to path algebras over a~simply laced Dynkin quiver.

\begin{theorem}[{\cite[Corollary~4.8]{Auslander74}}]
\label{DynkinBigModuleDecomposition}
Let $R$ be an artinian ring of finite representation type. Then every (possibly infinitely generated) $R$-module decomposes as a direct sum of finitely generated indecomposables, i.e.\ \(\Mod\dashmodule R=\Coprod(\ind\dashmodule R)\).    
\end{theorem}

\begin{lemma}
\label{BoundedDynkinCompletion}
Let \(A\) be a hereditary finite dimensional $K$-algebra of finite representation type.
Then for every metric on~\(\derived^b(\modf\dashmodule A)\) we have
\[
\mathfrak{S}'\big(\derived^b(\modf\dashmodule A)\big)\subseteq\derived^b(\modf\dashmodule A)
.\]
\end{lemma}

\begin{proof}
Since \(A\) is a hereditary ring, Proposition~\ref{BoundedHereditaryCompletion} gives us the inclusion of~categories \(\mathfrak{S}'\big(\derived^b(\modf\dashmodule A)\big)\subseteq\derived^b(\Mod\dashmodule A)\).
Let \(X\in\mathfrak{S}'\big(\derived^b(\modf\dashmodule A)\big)\) and let \(i\in\Z\).
By Theorem~\ref{DynkinBigModuleDecomposition}, we have \(H^i(X)\simeq \bigoplus_{M\in\ind\dashmodule A}M^{(J_M)}\) for some suitable indexing sets \(J_M\).
By Lemma~\ref{FinitenessOfIndexingSet}, the set \(J_M\) is finite for all \(M\in\ind\dashmodule A\), and~\(\ind\dashmodule A\) is also finite.
Hence \(H^i(X)\in\modf\dashmodule A\) and \(X\in\derived^b(\modf\dashmodule A)\).
\end{proof}

\begin{lemma}
\label{IntersectionIsTriangulated}
Let \(\mathcal{T}\) be a triangulated category with a good metric \(\{B_n\}_{n\in\N}\). Then \(\mathcal{B}:=\bigcap_{n=1}^\infty B_n\) is
a~triangulated subcategory of \(\mathcal{T}\).
\end{lemma}

\begin{proof}
Let \(T\in\mathcal{B}\) and let \(n\in\N\). Then \(T\in B_{n+1}\). Because the metric is good, we have the~inclusion \(\Sigma^{-1}B_{n+1}\cup B_{n+1}\cup \Sigma B_{n+1}\subseteq B_n\). Then \(\Sigma^{\pm1}T\in B_n\). Since $n$ was arbitrary, we get that \(\mathcal{B}\) is closed under shifts.
The subcategory \(\mathcal{B}\) then has to be a~triangulated subcategory as a subcategory closed under shifts and~extensions.
\end{proof}

Recall that a pair \((\mathcal{U},\mathcal{V})\) of subcategories of a triangulated category \(\mathcal{T}\) is called \textit{a~torsion pair} if \(\mathcal{V}\subseteq\mathcal{U}^{\perp}\) and \(\mathcal{U}*\mathcal{V}=\mathcal{T}\) (see \cite[Definition~2.2]{IyamaYoshino08}). In that case, it automatically holds that \(\mathcal{V}=\mathcal{U}^{\perp}\) and \(\mathcal{U}={}^{\perp}\mathcal{V}\).
We will make use of the~following well-known fact that, in our setting of derived categories of representation-finite algebras, every thick subcategory is a part of a torsion pair. 

\begin{lemma}
\label{OrthogonalTriangulatedSubcategory}
Let \(A\) be a hereditary finite dimensional $K$-algebra of finite representation type.
Let \(\mathcal{S}\) be a thick subcategory of \(\derived^b(\modf\dashmodule A)\).
Then there exists a~triangulated subcategory \(\mathcal{U}\) of \(\derived^b(\modf\dashmodule A)\), such that \((\mathcal{U},\mathcal{S})\) is a torsion pair.
\end{lemma}

\begin{proof}
As $A$ is hereditary and of finite representation type, the thick subcategory \(\mathcal{S}\) is a preenveloping class (even functorially finite) in \(\mathcal{T}\).
By \cite[Proposition~1.4]{Neeman10} and \cite[Theorem~2.9]{LakingVitoria20}, the subcategory \(\mathcal{S}\) is even enveloping.

Then by \cite[Proposition~2.4]{IyamaYoshino08}, the category \(\mathcal{S}\) is a right class of a torsion pair \((\mathcal{U},\mathcal{S})\) in~\(\derived^b(\modf\dashmodule A)\). 
\end{proof}

Finally, we get to the explicit computation of all completions of \(\derived^b(\modf\dashmodule A)\) for~$A$ a hereditary finite dimensional $K$-algebra of finite representation type.
If the~field $K$ is algebraically closed, every such a (connected) algebra $A$ is Morita equivalent to~$KQ$ for some simply laced Dynkin quiver $Q$
(see e.g.\ \cite[pg.\ 243]{AssemSimsonSkowronski06}).

\begin{theorem}
\label{DynkinCompletions}
Let \(A\) be a hereditary finite dimensional $K$-algebra of finite representation type.
Let \(\{B_n\}_{n\in\N}\) be a metric on~the~category \(\derived^b(\modf\dashmodule A)\).
Denote \(\mathcal{B}:=\bigcap_{n=1}^\infty B_n\).
Then
\[
\mathfrak{S}'\big(\derived^b(\modf\dashmodule A)\big)=\derived^b(\modf\dashmodule A)\cap\mathcal{B}[-\infty,\infty]^\perp\subseteq\derived(A)
.\]
Moreover, if \(\{B_n\}_{n\in\N}\) is good, then \(\mathfrak{S}'\big(\derived^b(\modf\dashmodule A)\big)=\derived^b(\modf\dashmodule A)\cap\mathcal{B}^\perp\).
\end{theorem}

\begin{proof}
``$\subseteq$''
The inclusion \(\mathfrak{S}'\big(\derived^b(\modf\dashmodule A)\big)\subseteq\derived^b(\modf\dashmodule A)\) follows from Lem\-ma~\ref{BoundedDynkinCompletion}. If \(X\in\mathfrak{S}'\big(\derived^b(\modf\dashmodule A)\big)\), then $X$ is compactly supported, so consequently for every \(j\in\Z\) there exists \(i\in\N\) with \(X\in \Sigma^jB_i^\perp\). Since \(\Sigma^j\mathcal{B}\subseteq \Sigma^jB_i\), we get that \(X\in\mathcal{B}[-\infty,\infty]^\perp\).

``$\supseteq$''
Let
\(X\in\derived^b(\modf\dashmodule A)\cap\mathcal{B}[-\infty,\infty]^\perp\). Then \(X\) is the homotopy colimit of~the~constant Cauchy sequence \(X\xrightarrow{\id_X}X\xrightarrow{\id_X}\cdots\) of elements of \(\derived^b(\modf\dashmodule A)\), so \(X\in\mathfrak{L}'\big(\derived^b(\modf\dashmodule A)\big)\).

Suppose that the cohomology of $X$ is concentrated in degrees \(-k+1,\ldots,k\) for~some \(k\in\N\).
Fix \(j\in\N_0\).
For any \(M\in\ind\dashmodule A\setminus \Summ(\mathcal{B}[-\infty,\infty])\) and any \(i\in\Z\), there exists \(n\in\N\) satisfying \(M\notin \Sigma^i\Summ(B_n)\).
Since \(j\) and \(k\) are fixed, and there are (up to isomorphism) only finitely many finitely generated indecomposable $A$\=/modules, we can actually find \(n\in\N\) big enough that
\[
M\notin \Summ(B_n)[-k-j,k+j]
\]
holds for~every \(M\in\ind\dashmodule A\setminus \Summ(\mathcal{B}[-\infty,\infty])\).

Let \(Z\in B_n\), and
denote \(P:=\coprod_{i=-k-j}^{k+j}\Sigma^{-i}H^i(Z)\).
By Theorem~\ref{DynkinBigModuleDecomposition}, all cohomologies of $P$ decompose as direct sums of indecomposable finitely generated $A$-modules. This yields \(P\in\Add(\mathcal{B}[-\infty,\infty])\) due to the choice of $n$.
But then
\[
\Hom_{\derived(A)}\left(\Sigma^jZ,X\right)=\Hom_{\derived(A)}\left(\Sigma^jP,X\right)=0
\]
because
\(X\in \mathcal{B}[-\infty,\infty]^{\perp}=\Add(\mathcal{B}[-\infty,\infty])^\perp.\)
Using \(X\in\Add(\mathcal{B}[-\infty,\infty])^\perp\) again, we also get
\[\Hom_{\derived(A)}\left(\Sigma^{-j}Z,X\right)=\Hom_{\derived(A)}\left(\Sigma^{-j}P,X\right)=0
.\]
Hence, we obtain \(X\in \Sigma^{\pm j}B_n^\perp\).

Since $j$ was picked arbitrarily, we conclude
\[
X\in\mathfrak{L}'\big(\derived^b(\modf\dashmodule A)\big)\cap\mathfrak{C}'\big(\derived^b(\modf\dashmodule A)\big)=\mathfrak{S}'\big(\derived^b(\modf\dashmodule A)\big)
.\]

For the proof of the final statement, we further assume that the metric \(\{B_n\}_{n\in\N}\) is good. Then \(\mathcal{B}\) is a triangulated subcategory of \(\derived^b(\modf\dashmodule A)\) by Lemma~\ref{IntersectionIsTriangulated}, so \(\mathcal{B}=\mathcal{B}[-\infty,\infty]\).
The claim now directly follows.
\end{proof}

The main message of Theorem~\ref{DynkinCompletions} is that for any (no matter how chaotic) metric \(\{B_n\}_{n\in\N}\) on~the~category \(\derived^b(\modf\dashmodule A)\), the completion \(\mathfrak{S}'\big(\derived^b(\modf\dashmodule A)\big)\) depends only on the intersection \(\mathcal{B}:=\bigcap_{n=1}^\infty B_n\) of all balls.
To the author's knowledge, other instances of classification results for completions in the literature involve much more well-behaved metrics; e.g.\ metrics determined by aisles (\cite[Theorem~1.1]{CummingsGratz24}), metrics determined by coaisles on discrete cluster categories (\cite[Theorem~1.2]{CummingsGratz24}), or metrics \(\left\{\langle G\rangle^{[-\infty,-n]}\right\}_{n\in\N}\) determined by shifts of a classical generator $G$ of a~triangulated category \(\mathcal{S}\) under a finiteness condition on the~relative finitistic dimension \(\FinitisticDimension(\mathcal{S}^{\opposite},G)<\infty\) (\cite[Theorem~3.11]{BiswasChenRahulParkerZheng24}).

Here, the relative finitistic dimension of a triangulated category \(\mathcal{S}\) at an object \(G\), defined in \cite[Definition~1.3]{BiswasChenRahulParkerZheng24} as
\[
\FinitisticDimension(\mathcal{S},G):=\inf\left\{n\in\N:G[-\infty,-1]^{\perp}\subseteq\Sigma^n\langle G\rangle^{[0,\infty]}\right\}
,\]
is a generalisation of notion (due to Bass~\cite{Bass60})
of the little finitistic dimension
\[
\FinitisticDimension(R):=\sup\left\{\ProjectiveDimension M:M\in\modf\dashmodule R, \ProjectiveDimension M<\infty\right\}
\]
of a ring $R$  in a sense that \(\FinitisticDimension\big(\derived(R)^c,R\big)=\FinitisticDimension(R)\) (see~\cite[Lemma~4.1.5]{BiswasChenRahulParkerZheng24}).

In this context, while the~algebra $A$ from Theorem~\ref{DynkinCompletions} itself serves as a classical (even strong) generator of the~category \(\derived^b(\modf\dashmodule A)\) and satisfies
\[
\FinitisticDimension\left(\derived^b(\modf\dashmodule A)^{\opposite},A\right)<\infty
\text{\quad and\quad}
\mathfrak{S}'\big(\derived^b(\modf\dashmodule A)\big)=\derived^b(\modf\dashmodule A)
\]
with respect to the metric \(\left\{\langle A\rangle^{[-\infty,-n]}\right\}_{n\in\N}\) as in \cite[Theorem~3.11]{BiswasChenRahulParkerZheng24}, our theorem also provides a description of completions with respect to metrics not related to the strong generator - beyond other things, metrics determined by shifts of any (non-generating) object $M$ with \(\FinitisticDimension\left(\derived^b(\modf\dashmodule A)^{\opposite},M\right)=\infty\).

There exist, however, cases where the methods from Theorem~\ref{DynkinCompletions}, \cite{BiswasChenRahulParkerZheng24} and \cite{CummingsGratz24} all apply at the same time. We offer an insight in the following remark.

\begin{remark}
Let \(A\) be a hereditary finite dimensional $K$-algebra of finite representation type, and let \((\mathcal{U},\mathcal{V})\) be a $t$-structure on \(\derived^b(\modf\dashmodule A)\).
The $t$-structure determines an aisle metric \(\{\Sigma^n\mathcal{U}\}_{n\in\N}\) (see \cite[Definition~4.2]{CummingsGratz24}) on \(\derived^b(\modf\dashmodule A)\), and~the~metric is good.
Theorem~\ref{DynkinCompletions} then yields
\[
\mathfrak{S}'\big(\derived^b(\modf\dashmodule A)\big)=\left(\bigcap_{n=1}^\infty \Sigma^n\mathcal{U}\right)^\perp=\bigcup_{n=1}^\infty\Sigma^{-n}\mathcal{V}\subseteq\derived^b(\modf\dashmodule A)
\]
giving us a specific instance of a general result \cite[Theorem~1.1]{CummingsGratz24} about completions with respect to aisle metrics being unions of shifts of~coaisles.

In particular, if the $t$-structure \((\mathcal{U},\mathcal{V})\) is bounded below, we have \(\bigcap_{n=1}^\infty \Sigma^n\mathcal{U}=0\) and~\(\mathfrak{S}'\big(\derived^b(\modf\dashmodule A)\big)=\bigcup_{n=1}^\infty\Sigma^{-n}\mathcal{V}=\derived^b(\modf\dashmodule A)\) as noticed in \cite[Corollary~3.6]{BiswasChenRahulParkerZheng24} and implicitly proven in \cite[Theorem~5.1]{BiswasChenRahulParkerZheng23} in the first version of said paper \cite{BiswasChenRahulParkerZheng24}.
\end{remark}

\begin{corollary}
\label{ThickSubcategories}
Let \(A\) be a hereditary finite dimensional $K$-algebra of finite representation type.
Then completions of \(\derived^b(\modf\dashmodule A)\) are precisely the thick subcategories of \(\derived^b(\modf\dashmodule A)\).
\end{corollary}

\begin{proof}
Let \(\mathfrak{S}'\big(\derived^b(\modf\dashmodule A)\big)\) be a completion with respect to a metric \(\{B_n\}_{n\in\N}\).
If we denote \(\mathcal{B}:=\bigcap_{n=1}^\infty B_n\), by Theorem~\ref{DynkinCompletions} we get
\[
\mathfrak{S}'\big(\derived^b(\modf\dashmodule A)\big)=\derived^b(\modf\dashmodule A)\cap\mathcal{B}[-\infty,\infty]^\perp
.\]
In other words, the completion \(\mathfrak{S}'\big(\derived^b(\modf\dashmodule A)\big)\) is Hom-orthogonal to the sus\-pen\-sion closed subcategory \(\mathcal{B}[-\infty,\infty]\) of \(\derived^b(\modf\dashmodule A)\). Then \(\mathfrak{S}'\big(\derived^b(\modf\dashmodule A)\big)\) must necessarily be a triangulated subcategory of~\(\derived^b(\modf\dashmodule A)\), which is closed under direct summands. Hence, it is a thick subcategory.

Conversely, let \(\mathcal{S}\) be a thick subcategory of \(\derived^b(\modf\dashmodule A)\).
Lemma~\ref{OrthogonalTriangulatedSubcategory} gives us a~triangulated subcategory \(\mathcal{U}\subseteq\derived^b(\modf\dashmodule A)\) such that \((\mathcal{U},\mathcal{S})\) is a torsion pair, implying \(\mathcal{S}=\mathcal{U}^{\perp}\).
The constant sequence of subcategories \(\{\mathcal{U}\}_{n\in N}\) is a good metric on \(\derived^b(\modf\dashmodule A)\).
By Theorem~\ref{DynkinCompletions}, the~completion with respect to this metric is \(\mathcal{S}\).
\end{proof}

\begin{example}
Let $K$ be a field. There are only two completions of \(\derived^b(\modf\dashmodule K)\), namely \(\derived^b(\modf\dashmodule K)\) and $0$.

There are only five completions of \(\derived^b(\modf\dashmodule KA_2)\), namely all the thick subcategories \(\derived^b(\modf\dashmodule KA_2)\), \(\langle\simple(1)\rangle\), \(\langle\simple(2)\rangle\), \(\langle\PP(2)\rangle\), and $0$ (see e.g.\ \cite[Example~5.1.3]{GratzStevenson23}).
\end{example}

We finish the section by a non-example.
\begin{example}
\label{HereditaryNonExample}
Theorem~\ref{DynkinCompletions} does not hold for a general hereditary ring. Let us set \(\mathcal{T}:=\derived^b(\modf\dashmodule\Z)\) and let \(\mathcal{S}\) be the triangulated subcategory of \(\mathcal{T}\) generated by finitely generated torsion abelian groups. We consider the constant metric \(\{\mathcal{S}\}_{n\in\N}\) on \(\mathcal{T}\). We claim that \(\mathfrak{S}'(\mathcal{T})=\langle\Q\rangle\subseteq\derived(\Mod\dashmodule\Z)\).
Thanks to Proposition~\ref{BoundedHereditaryCompletion}, we know so far that \(\mathfrak{S}'(\mathcal{T})\subseteq\derived^b(\Mod\dashmodule\Z)\).

Let \(E=\hoco E_*\in\mathfrak{S}'(\mathcal{T})\) for some good Cauchy sequence $E_*$ in~\(\mathcal{T}\). Let \(i\in\Z\) and \(M:=H^i(E)\in\Mod\dashmodule\Z\). Then \(\Sigma^{-i}M\) is compactly supported as a~direct summand of \(E\in\mathfrak{C}'(\mathcal{T})\). Hence \(\Hom_{\derived(\Mod\dashmodule\Z)}(\mathcal{S},\Sigma^{-i}M)=0\). We infer that \(\Hom_{\Z}\big(\bigoplus_{n=2}^\infty\Z/n\Z,M\big)\simeq0\simeq\Ext^1_{\Z}\big(\bigoplus_{n=2}^\infty\Z/n\Z,M\big)\), so $M$ is torsion-free and~injective (by Baer's criterion \cite[Theorem~1]{Baer40}, see also \cite[Lemma~18.3]{AndersonFuller92} or \cite[Proposition~1.2]{Trlifaj96}). Therefore \(M\simeq \Q^{(J)}\) for some indexing set $J$.

To show that $J$ is finite, we use an argument analogous to the proof of Lemma~\ref{FinitenessOfIndexingSet}.
Since \(\Sigma^{-i}\Q\in\derived(\Mod\dashmodule\Z)\) is compactly supported, Lemma~\ref{BasicExtensionFactorisationProperty} gives us \(N\in\N\) such that the map
\[
\Hom_{\derived(\Mod\dashmodule\Z)}\left(\varphi,\Sigma^{-i}\Q\right):\Hom_{\derived(\Mod\dashmodule\Z)}\left(E,\Sigma^{-i}\Q\right)\rightarrow\Hom_{\derived(\Mod\dashmodule\Z)}\left(E_N,\Sigma^{-i}\Q\right)
\]
is an abelian group isomorphism (so a bijection), where \(\varphi:E_N\rightarrow E\) is the canonical map.
But \(\Hom_{\derived(\Mod\dashmodule\Z)}\left(E_N,\Sigma^{-i}\Q\right)\) is at most countable as a set. If $J$ was infinite, we would get that
\[\Hom_{\derived(\Mod\dashmodule\Z)}\left(E,\Sigma^{-i}\Q\right)\simeq\Hom_{\derived(\Mod\dashmodule\Z)}\left(\Sigma^{-i}\Q^{(J)},\Sigma^{-i}\Q\right)\simeq \Q^J\]
is uncountable, which is impossible. Hence $J$ is finite.

So far we have shown that \(\mathfrak{S}'(\mathcal{T})\subseteq\langle\Q\rangle\). To prove the equality, it is enough to~observe that \(\Q\in\mathfrak{L}'(\mathcal{T})\).
We look at the following diagram in \(\modf\dashmodule\Z\)
\[
C_*:= \Z \xhookrightarrow{2\cdot\blank} \Z \xhookrightarrow{3\cdot\blank} \Z \xhookrightarrow{4\cdot\blank} \Z \xhookrightarrow{} \cdots
\]
with \(\varinjlim C_*=\Q\) in \(\Mod\dashmodule\Z\).
We have for all \(i\in\N\) that \(\Coker(i\cdot\blank)\) is a torsion group, so if we view \(\Z\) as a complex concentrated in degree $0$ and interpret $C_*$ as a~diagram in~\(\mathbf{D}^b(\modf\dashmodule \Z)\), then \(C_*\) is Cauchy with \(\hoco C_*=\Q\) in \(\derived(\Mod\dashmodule\Z)\). We conclude \(\mathfrak{S}'(\mathcal{T})=\langle\Q\rangle\).
\end{example}

\section{Functors between completions}
\label{sec:InducedFunctors}

This section is divided into three parts. In Subsection~\ref{sub:InducedFunctorsOnCompletions}, we elaborate on the~theory of induced functors on completions from \cite{SunZhang21}.
Subsection~\ref{sec:PreimagesImages} contains the~main new result of this section (Theorem~\ref{CalculatingCompletionsElsewhere}) - we introduce the preimage and image metrics on triangulated categories and use them to induce equivalences between their completions.
We give an application of the result for categories with Serre functors in Subsection~\ref{sub:SerreFunctors}.

\subsection{Induced functors on completions}
\label{sub:InducedFunctorsOnCompletions}

In this subsection, we recall one of the~main results from \cite{SunZhang21} by Sun and Zhang, namely Theorem~1.1 (later stated as Theorem~4.3). They prove that an adjoint pair of triangulated functors between triangulated categories equipped with metrics induces a triangulated functor between the respective completions, assuming that the adjoint functors are, in some sense, ``continuous'' with respect to the metrics. The proper term used by Sun and Zhang is that of a compression functor.

Since the original theorem deals with good metrics only, we reprove it here in~the~full generality of (potentially) non-good metrics. 
The proofs in \cite{SunZhang21} require only minimal adjustments to include the stronger hypothesis, but we repeat them for the~convenience of~the~reader.

\begin{definition}[{\cite[Definition~4.1]{SunZhang21}}]
Let $\mathcal{T}$, \(\mathcal{S}\) be triangulated categories with metrics \(\{B_n\}_{n\in\N}\) and \(\{C_n\}_{n\in\N}\), respectively. We say that a triangulated functor \(F:\mathcal{T}\rightarrow\mathcal{S}\) is \textit{a compression} if for every \(n\in\N\) there exists \(m\in\N\) such that \(F(B_m)\subseteq C_n\).    
\end{definition}

\begin{remark}
Let \(\{B_n\}_{n\in\N}\) and \(\{C_n\}_{n\in\N}\) be metrics on a triangulated category \(\mathcal{T}\). Notice that \(\{B_n\}_{n\in\N}\leq\{C_n\}_{n\in\N}\) if and only if the identity functor \(\id:\mathcal{T}\rightarrow\mathcal{T}\) from \(\mathcal{T}\) equipped with the metric \(\{B_n\}_{n\in\N}\) to \(\mathcal{T}\) equipped with the metric \(\{C_n\}_{n\in\N}\) is a compression.
\end{remark}

\begin{lemma}[{\cite[Lemma~4.2]{SunZhang21}}]
\label{CompressionsPreserveCauchySequences}
Let $\mathcal{T}$, \(\mathcal{S}\) be triangulated categories with metrics \(\{B_n\}_{n\in\N}\) and \(\{C_n\}_{n\in\N}\), respectively, and let \(F:\mathcal{T}\rightarrow\mathcal{S}\) be a compression. Then $F$ preserves (good) Cauchy sequences. 
\end{lemma}

\begin{proof}
We prove the statement only for good Cauchy sequences, because the other case is analogous.

Let \(E_*:= E_1\rightarrow E_2 \rightarrow \cdots\) be a good Cauchy sequence in \(\mathcal{T}\).
We check that its image \(F(E_*):= F(E_1)\rightarrow F(E_2) \rightarrow \cdots\) is a good Cauchy sequence in~\(\mathcal{S}\).
Let \(i\in\N\) and \(j\in\Z\). We find \(k\in\N\) such that \(F(B_k)\subseteq C_i\). We find \(M\in\N\) such that for~every \(m'\geq m\geq M\) the triangle \(E_m\rightarrow E_{m'}\rightarrow D_{m,m'}\rightarrow\Sigma E_m\) satisfies \(\Sigma^jD_{m,m'}\in B_k\).
We have that \(F(E_m)\rightarrow F(E_m')\rightarrow F(D_{m,m'})\rightarrow\Sigma F(E_m)\) is also a~triangle because $F$ is a triangulated functor. Since \(\Sigma^jD_{m,m'}\in B_k\), it holds that \(F(\Sigma^jD_{m,m'})\simeq\Sigma^jF(D_{m,m'})\in C_i\). Thus, the sequence \(F(E_*)\) is good Cauchy.
\end{proof}

\begin{lemma}[{\cite[Theorem~4.3, pt.\ 2]{SunZhang21}}]
\label{CompressionsPreserveCompactlySupportedElements}
Let $\mathcal{T}$, \(\mathcal{S}\) be triangulated categories with metrics \(\{B_n\}_{n\in\N}\) and \(\{C_n\}_{n\in\N}\), respectively, and let \(F:\mathcal{T}\rightarrow\mathcal{S}\) be a compression.
Then \(\blank\circ F:\Mod\dashmodule\mathcal{S}\rightarrow\Mod\dashmodule\mathcal{T}\) induces functors
\[
\blank\circ F:\mathfrak{C}({\mathcal{S}})\rightarrow\mathfrak{C}(\mathcal{T})
\text{\ \ \ \ and \ \ \ }
\blank\circ F:\mathfrak{WC}({\mathcal{S}})\rightarrow\mathfrak{WC}(\mathcal{T})
.\]
\end{lemma}

\begin{proof}
We prove the statement only for compactly supported elements, because the~other case is analogous.

Let \(G\in\mathfrak{C}(\mathcal{S})=\bigcap_{j\in\Z}\bigcup_{i\in\N}Y(\Sigma^j C_i)^{\perp}\). Let \(j\in\Z\). We find \(i,k\in\N\) such that \(G\in Y(\Sigma^j C_i)^{\perp}\) and \(F(B_k)\subseteq C_i\). Let \(X\in B_k\).
By Yoneda's lemma, we have
\[
\begin{split}
\Hom\left(Y(\Sigma^j X),G\circ F\right)&\simeq G\left(F(\Sigma^jX)\right)\simeq\Hom\Big(Y\left(F(\Sigma^jX)\right),G\Big)\\
 &\simeq\Hom\Big(Y\big(\Sigma^jF(X)\big),G\Big)=0
\end{split}
\]
since \(F(X)\in C_i\) and \(G\in Y(\Sigma^j C_i)^{\perp}\).

Because $j$ was arbitrary, we obtain \(G\circ F\in\mathfrak{C}(\mathcal{T})=\bigcap_{j\in\Z}\bigcup_{i\in\N}Y(\Sigma^j B_i)^{\perp}\).
\end{proof}

\begin{lemma}[{\cite[Theorem~4.3, pt.\ 1]{SunZhang21}}]
\label{CompressionsPreservePrecompletion}
Let \(F:\mathcal{T}\rightleftarrows \mathcal{S}:G\) be adjoint triangulated functors between triangulated categories with metrics \(\{B_n\}_{n\in\N}\) and \(\{C_n\}_{n\in\N}\), respectively.
If $G$ is a compression, then \(\blank\circ F:\Mod\dashmodule\mathcal{S}\rightarrow\Mod\dashmodule\mathcal{T}\) induces a functor \(\blank\circ F:\mathfrak{L}({\mathcal{S}})\rightarrow\mathfrak{L}(\mathcal{T})\).

In particular, if \(E=\moco E_*\in\mathfrak{L}({\mathcal{S}})\) for some good Cauchy sequence \(E_*\) in~\(\mathcal{S}\), then \(E\circ F\simeq\moco G(E_*)\).
\end{lemma}

\begin{proof}
Let \(E\simeq \moco E_*\in\mathfrak{L}({\mathcal{S}})\) be as in the statement of the lemma.
Then by~adjunction, we obtain
\[
\begin{split}
E\circ F&= \left(\moco E_*\right)\circ F\simeq\varinjlim \Hom_{\mathcal{S}}\big(F(\blank),E_*\big)\simeq\varinjlim\Hom_{\mathcal{T}}\big(\blank,G(E_*)\big) \\
&=\moco G(E_*),
\end{split}
\]
where \(\moco G(E_*)\in\mathfrak{L}(\mathcal{T})\) because the compression functor $G$ preserves good Cauchy sequences by Lemma~\ref{CompressionsPreserveCauchySequences}.
\end{proof}

We will heavily rely on the following well-known category-theoretical fact and~its consequences.

\begin{proposition}[{\cite[Theorem~IV.3.1]{Lane98}}]
\label{FullyFaithfulAdjoints}
Let \(F:\mathcal{A}\rightleftarrows \mathcal{B}:G\) be adjoint functors. Let \(\eta:1_{\mathcal{A}}\Rightarrow GF\) and~\(\varepsilon:FG\Rightarrow 1_{\mathcal{B}}\) be the unit and the counit of the adjunction, respectively.
Then $F$ is fully faithful if and only if $\eta$ is a natural isomorphism. In this case, $G$ is essentially surjective.

Dually, $G$ is fully faithful if and only if \(\varepsilon\) is a natural isomorphism. And in this case, $F$ is essentially surjective.
\end{proposition}

For the next corollary, we recall that a functor \(F:\mathcal{A}\rightarrow\mathcal{B}\) is \textit{conservative} provided that
whenever \(F(f)\) is an~isomorphism in \(\mathcal{B}\) for a morphism $f$ in \(\mathcal{A}\), then $f$ has to be an isomorphism as well.

\begin{corollary}
\label{FullyFaithfulAdjointConservative}
Let \(F:\mathcal{A}\rightleftarrows \mathcal{B}:G\) be adjoint functors.
If one of the functors $F$, \(G\) is conservative and the other one is fully faithful, then \(F\) and $G$ are mutually inverse equivalences of categories.
\end{corollary}

\begin{proof}
By symmetry, it is sufficient to assume that $F$ is conservative and $G$ is fully faithful.
Let \(\eta:1_{\mathcal{A}}\Rightarrow GF\) and \(\varepsilon:FG\Rightarrow 1_{\mathcal{B}}\) be the unit and the counit of~the~adjunction, respectively.

Let \(A\in\mathcal{A}\).
By the adjunction, the identity map \(1_{F(A)}\) decomposes as
\[
F(A)\xrightarrow{F(\eta_A)}FGF(A)\xrightarrow{\varepsilon_{F(A)}}F(A)
.\]
By Proposition~\ref{FullyFaithfulAdjoints}, the map \(\varepsilon_{F(A)}\) is an isomorphism, so \(F(\eta_A)\) is an isomorphism as well. 
Since $F$ is conservative, we have that \(\eta_A:A\rightarrow GF(A)\) is an isomorphism.
As $A$ was arbitrary, we obtain that $G$ is essentially surjective.

We conclude that $G$ is an equivalence of categories. By the uniqueness of adjoints, $F$ is its inverse.
\end{proof}

\begin{corollary}[{\cite[Corollary~3.2]{SunZhang21}}]
\label{InheritanceOfFullFaithfulness}
Let \(F:\mathcal{A}\rightleftarrows \mathcal{B}:G\) be adjoint additive functors between additive categories.
Then \(\blank\circ F:\Mod\dashmodule\mathcal{B}\rightleftarrows \Mod\dashmodule\mathcal{A}:\blank\circ G\) is an~adjoint pair.

Furthermore,
if $G$ is fully faithful, then \(\blank\circ F\) is fully faithful.
Conversely, if $F$ is fully faithful, then \(\blank\circ G\) is fully faithful.
\end{corollary}

Finally, we can present the theorem by Sun and Zhang about induced functors.
As mentioned above, the statement presented here is more general than in~the~original article because it also concerns metrics which are not good.

\begin{theorem}[{\cite[Theorem~4.3, pt.\ 3]{SunZhang21}}]
\label{SunZhangTheorem}
Let \(F:\mathcal{T}\rightleftarrows \mathcal{S}:G\) be adjoint triangulated functors between triangulated categories with metrics \(\{B_n\}_{n\in\N}\) and \(\{C_n\}_{n\in\N}\), respectively.
If $F$ and $G$ are compressions, then the functor \(\blank\circ F:\Mod\dashmodule\mathcal{S}\rightarrow\Mod\dashmodule\mathcal{T}\) induces a triangulated functor \(\blank\circ F:\mathfrak{S}({\mathcal{S}})\rightarrow\mathfrak{S}(\mathcal{T})\).
\end{theorem}

\begin{proof}
By Lemma~\ref{CompressionsPreserveCompactlySupportedElements} and Lemma~\ref{CompressionsPreservePrecompletion}, we get that the functor \(\blank\circ F:\mathfrak{S}({\mathcal{S}})\rightarrow\mathfrak{S}(\mathcal{T})\) is well-defined. We only need to check that it is triangulated.

By Theorem~\ref{NeemansCompletionTheorem}, we know that \(\mathfrak{S}({\mathcal{S}}),\mathfrak{S}(\mathcal{T})\) are triangulated categories with shifts given by 
\(\mathbf{\Sigma} E:=E\circ{\Sigma}^{-1}\)
for either \(E\in\mathfrak{S}({\mathcal{S}})\), or \(E\in\mathfrak{S}({\mathcal{T}})\), respectively.
(For~clarity, in~this proof we denote shifts in completions by bold \(\mathbf{\Sigma}\) instead of \(\Sigma\).).

We define a natural isomorphism \(\psi:(\blank\circ F) \circ \mathbf{\Sigma}\Rightarrow\mathbf{\Sigma}\circ(\blank\circ F)\) with components \(\psi_E\) given for each \(E\in\mathfrak{S}(\mathcal{S})\) by~the~equation
\[
\begin{split}
[(\blank\circ F) \circ \mathbf{\Sigma}](E)&=(\blank\circ F)\big(E\circ{\Sigma}^{-1}\big)= E\circ{\Sigma}^{-1}\circ F\simeq E\circ F\circ{\Sigma}^{-1}\\
&=\mathbf{\Sigma}(E\circ F)=[\mathbf{\Sigma}\circ(\blank\circ F)](E)
\end{split}
\]
where the existence of the natural isomorphism in the middle follows from the fact that $F$ is a~triangulated functor.

We fix a triangle in \(\mathfrak{S}({\mathcal{S}})\), which is necessarily isomorphic to
\[
\moco A_*\xrightarrow{a} \moco B_*\xrightarrow{b}\moco C_*\xrightarrow{c}\mathbf{\Sigma}\left(\moco A_*\right)
\]
for a good Cauchy sequence of triangles \(A_*\rightarrow B_*\rightarrow C_* \rightarrow \Sigma A_*\) in \(\mathcal{S}\).

Under the notation for natural transformations from Subsection~\ref{GeneralCategoryTheory} and after setting \(\delta:=\psi_{\begin{picture}(40,10)
\put(19,3){\makebox(0,0)[b]{\text{\rm \footnotesize Mocolim}}}
\put(4,1){\ \vector(1,0){25}}
\put(37,4.5){\({}_{A_*}\)}
\end{picture}\,\,}\circ c_F\)
, the candidate triangle 
\[
\begin{split}
\left(\moco A_*\right)\circ F&\xrightarrow{a_F} \left(\moco B_*\right)\circ F\xrightarrow{b_F}\\
&\xrightarrow{b_F}\left(\moco C_*\right)\circ F\xrightarrow{\delta} \mathbf{\Sigma}\left[\left(\moco A_*\right)\circ F\right]
\end{split}
\]
is isomorphic (as a candidate triangle) to the distinguished triangle
\[
\moco G(A_*)\rightarrow \moco G(B_*)\rightarrow\moco G(C_*)\rightarrow\mathbf{\Sigma}\big(\moco G(X_*)\big),
\]
by Lemma~\ref{CompressionsPreservePrecompletion} and the fact that $G$ is a compression. Hence, it is a distinguished triangle itself. Consequently, the functor \(\blank\circ F\) is triangulated.
\end{proof}

\begin{corollary}[{\cite[Corollary on pg.\ 102]{SunZhang21}}]
Let \(F:\mathcal{T}\rightleftarrows \mathcal{S}:G\) be adjoint triangulated functors between triangulated categories with metrics \(\{B_n\}_{n\in\N}\) and~\(\{C_n\}_{n\in\N}\), respectively.
If $F$ and $G$ are compressions and $G$ has a right adjoint which is a~compression, then the triangulated functors \(\blank\circ F:\mathfrak{S}(\mathcal{S})\rightleftarrows \mathfrak{S}(\mathcal{T}):\blank\circ G\) form an~adjoint pair.
\end{corollary}

\begin{proof}
The right adjoint to the triangulated functor $G$ is automatically triangulated by \cite[Lemma~5.3.6]{Neeman01}.
The rest then straightforwardly follows from Theorem~\ref{SunZhangTheorem} and~Corollary~\ref{InheritanceOfFullFaithfulness}.
\end{proof}

\subsection{Preimages and images of metrics}
\label{sec:PreimagesImages}
This subsection begins with a brief overview. Precise definitions and statements follow.

Assume we have a triangulated functor \(F:\mathcal{T}\rightarrow\mathcal{S}\). Given a metric on \(\mathcal{S}\), we construct a so-called preimage metric on \(\mathcal{T}\) rendering the functor $F$ a compression with respect to these two metrics. Conversely, we want to define an image of a~metric on \(\mathcal{T}\) under \(\mathcal{S}\) - however, this construction is only possible while $F$ is full.

We then use the image and preimage metrics together with the theory of induced functors between completions from \cite{SunZhang21} recalled in Subsection~\ref{sub:InducedFunctorsOnCompletions} to prove that under certain condition on an existence of sufficiently nice adjoints we can calculate a~completion of a triangulated category as a completion of a different triangulated category with respect to a preimage metric (see Theorem~\ref{CalculatingCompletionsElsewhere}).
The conditions for~the~proposition to work are satisfied in the presence of a Serre functor (see Subsection~\ref{sub:SerreFunctors}).

\begin{definition}
Let $\mathcal{T}$, \(\mathcal{S}\) be triangulated categories and \(F:\mathcal{T}\rightarrow\mathcal{S}\) be a functor.
Let \(\{B_n\}_{n\in\N}\) be a metric on \(\mathcal{S}\). We define \textit{the preimage of} \(\{B_n\}_{n\in\N}\) \textit{under} \(F\) as the sequence \(\{F^{-1}(B_n)\}_{n\in\N}\) of subcategories of \(\mathcal{T}\).
\end{definition}

\begin{lemma}
\label{PreimageIsAMetric}
Let $\mathcal{T}$, \(\mathcal{S}\) be triangulated categories and \(F:\mathcal{T}\rightarrow\mathcal{S}\) be a triangulated functor with its natural isomorphism \(\varphi:F\Sigma\Rightarrow\Sigma F\).
If \(\{B_n\}_{n\in\N}\) is a metric on~\(\mathcal{S}\), then \(\{F^{-1}(B_n)\}_{n\in\N}\) is a metric on \(\mathcal{T}\).
Furthermore, if \(\{B_n\}_{n\in\N}\) is good, so is \(\{F^{-1}(B_n)\}_{n\in\N}\).
\end{lemma}

\begin{proof}
Since \(0\in B_n\) for all \(n\in\N\) and $F$ is additive, we have \(0\in F^{-1}(B_n)\).
Also \(F^{-1}(B_1)\supseteq F^{-1}(B_2)\supseteq F^{-1}(B_3)\supseteq\ldots\) is a non-increasing chain of subcategories.

We need to check that \(F^{-1}(B_n)* F^{-1}(B_n)= F^{-1}(B_n)\) for all \(n\in\N\).
Let \(n\in\N\). We fix a~triangle \(U\xrightarrow{f}V\xrightarrow{g}W\xrightarrow{h}\Sigma U\) in \(\mathcal{T}\) with \(U,W\in F^{-1}(B_n)\).
We know that
\[
F(U)\xrightarrow{F(f)}F(V)\xrightarrow{F(g)}F(W)\xrightarrow{\varphi_U\circ F(h)}\Sigma F(U)
\]
is a triangle
because \(F\) is triangulated.
As \(F(U),F(W)\in B_n\), so is \(F(V)\) since \(B_n\) is extension-closed. Hence \(V\in F^{-1}(B_n)\), and \(F^{-1}(B_n)\) is also extension-closed.
This finishes the proof of~the~fact that the preimage of a metric is a metric.

Now for the proof of the second statement, assume that the metric \(\{B_n\}_{n\in\N}\) is good.
In order to verify that the metric \(\{F^{-1}(B_n)\}_{n\in\N}\) is also good, we need to~check
\[
\Sigma^{-1}F^{-1}(B_{n+1})\cup F^{-1}(B_{n+1})\cup \Sigma F^{-1}(B_{n+1})\subseteq F^{-1}(B_n)
\]
for all \(n\in\N\). Fix \(n\in\N\).
We already know \(F^{-1}(B_{n+1})\subseteq F^{-1}(B_n)\).
Let us have \(X\in F^{-1}(B_{n+1})\). Because \(F(X)\in B_{n+1}\), we get \(\Sigma F(X)\in B_n\). However, there is an isomorphism \(\varphi_X:F\Sigma(X)\rightarrow \Sigma F(X)\), and since \(B_n\) is extension-closed (thus closed under isomorphisms), it holds that \(\Sigma X\in  F^{-1}(B_n)\).
The proof that
\(\Sigma^{-1} X\in  F^{-1}(B_n)\) is analogous.
\end{proof}

\begin{definition}
Let $\mathcal{T}$, \(\mathcal{S}\) be triangulated categories and \(F:\mathcal{T}\rightarrow\mathcal{S}\) be a functor.
Let \(\{B_n\}_{n\in\N}\) be a metric on \(\mathcal{T}\). We define \textit{the image of} \(\{B_n\}_{n\in\N}\) \textit{under} \(F\) as the~sequence \(\{F(B_n)\}_{n\in\N}\), where \(F(B_n)\) denotes the subcategory
\[
F(B_n)=\{Z\in\mathcal{S}:Z\simeq F(X),X\in B_n\} \text{\ \ \ for any\ \ \ } n\in\N
.\]
\end{definition}

\begin{lemma}
\label{ImageIsSometimesAMetric}
Let $\mathcal{T}$, \(\mathcal{S}\) be triangulated categories and \(F:\mathcal{T}\rightarrow\mathcal{S}\) be a full triangulated functor with its natural isomorphism \(\varphi:F\Sigma\Rightarrow\Sigma F\).
If \(\{B_n\}_{n\in\N}\) is a~metric
on \(\mathcal{S}\), then \(\left\{{F(B_n)}\right\}_{n\in\N}\) is a~metric
on \(\mathcal{T}\).

Furthermore, if \(\{B_n\}_{n\in\N}\) is a~good metric
on \(\mathcal{S}\), then \(\left\{{F(B_n)}\right\}_{n\in\N}\) is a~good metric
on \(\mathcal{T}\).
\end{lemma}

\begin{proof}
Since \(0\in B_n\) for all \(n\in\N\) and $F$ is additive, we have \(0\in {F(B_n)}\).
Also \(\mathcal{S}\supseteq F(B_2)\supseteq F(B_3)\supseteq\ldots\) is a~non-increasing chain of subcategories.

We need to check that \(F(B_n)* F(B_n)= F(B_n)\) for all \(n\in\N\).
Let \(n\in\N\). We fix a~triangle in \(\mathcal{S}\) of the form
\[
F(X)\xrightarrow{}V\xrightarrow{}F(Z)\xrightarrow{w}\Sigma F(X)
\]
for some \(X,Z\in B_n\).
Using the fullness of $F$, we find a map
\[
h:Z\rightarrow\Sigma X \text{ satisfying } F(h)=\varphi^{-1}_X\circ w:F(Z)\rightarrow F(\Sigma X)
.\]
We complete $h$ to a triangle 
\(X\xrightarrow{f}U\xrightarrow{g}Z\xrightarrow{h}\Sigma X\).
Since \(B_n*B_n=B_n\), we have \(U\in B_n\).
We know that
\[
F(X)\xrightarrow{F(f)}F(U)\xrightarrow{F(g)}F(Z)\xrightarrow{\varphi_X\circ F(h)}\Sigma F(X)
\]
is a triangle
because \(F\) is triangulated.
And since \(\varphi_X\circ F(h)=w\), this triangle is isomorphic to the triangle \(F(X)\xrightarrow{}V\xrightarrow{}F(Z)\xrightarrow{w}\Sigma F(X)\).
In particular, it holds that \(V\simeq F(U)\in F(B_n)\), so \(F(B_n)\) is extension closed.

Now furthermore assume that the metric \(\{B_n\}_{n\in\N}\) is good.
In order to verify that the metric \(\left\{{F(B_n)}\right\}_{n\in\N}\) is also good, we need to check
\[
\Sigma^{-1}F(B_{n+1})\cup F(B_{n+1})\cup \Sigma F(B_{n+1})\subseteq F(B_n)
\]
for all \(n\in\N\). Fix \(n\in\N\).
We already know that it holds that \(F(B_{n+1})\subseteq F(B_n)\). Let \(F(X)\in F(B_{n+1})\) for some \(X\in B_{n+1}\). We get \(\Sigma X\in B_n\). However, there is an~isomorphism \(\varphi_X:F\Sigma(X)\rightarrow \Sigma F(X)\).
Thus \(\Sigma F(X)\simeq F\Sigma(X)\in F(B_n)\).
The~proof that
\(\Sigma^{-1} X\in  F^{-1}(B_n)\) is analogous.
\end{proof}

\begin{example}
The statement of Lemma~\ref{ImageIsSometimesAMetric} does hold in general without the~assumption on the functor \(F:\mathcal{T}\rightarrow\mathcal{S}\) being full.

Let \(K_1\) and \(K_2\) be two copies of the same field $K$ and consider the coproduct category \(\mathcal{T}:=\derived^b(\modf\dashmodule K_1)\coprod\derived^b(\modf\dashmodule K_2)\).
We set \(\mathcal{S}:=\derived^b(\modf\dashmodule KA_2)\) to be the~bounded derived category of the path algebra of the Dynkin quiver \(A_2\). The~notation \(S(1)\) and \(S(2)\) will stand for the two non-isomorphic simple modules in the module category \(\modf\dashmodule KA_2\).

Then the triangulated functor \(F:\mathcal{T}\rightarrow\mathcal{S}\) given by \(K_1\mapsto S(1)\) and \(K_2\mapsto S(2)\) is not full, and \(F(\mathcal{T})\subsetneq\mathcal{S}\), but \(F(\mathcal{T})*F(\mathcal{T})=\mathcal{S}\).
If we interpret the category \(\mathcal{T}\) as a constant metric \(\{\mathcal{T}\}_{n\in\N}\) on itself, then the image \(\{F(\mathcal{T})\}_{n\in\N}\) of \(\{\mathcal{T}\}_{n\in\N}\) is not a~metric on \(\mathcal{S}\) at all because \(F(\mathcal{T})\) is not extension-closed.
\end{example}

We prove a relation between a preimage metric under a functor and an image metric under a fully faithful adjoint.

\begin{lemma}
\label{AdjointsImageLiesInPreimage}
Let \(F:\mathcal{T}\rightleftarrows \mathcal{S}:G\) be adjoint triangulated functors between triangulated categories.
Let \(\{B_n\}_{n\in\N}\) be a metric on \(\mathcal{T}\).
If $F$ is fully faithful, then \(\left\{{F(B_n)}\right\}_{n\in\N}\leq\{G^{-1}(B_n)\}_{n\in\N}\).
\end{lemma}

\begin{proof}
We know that \(\left\{{F(B_n)}\right\}_{n\in\N}\) and \(\{G^{-1}(B_n)\}_{n\in\N}\) are metrics by Lemma~\ref{ImageIsSometimesAMetric} and Lemma~\ref{PreimageIsAMetric}, respectively.
For all \(n\in\N\) we have \(GF(B_n)=B_n\) because of~the~natural isomorphism \(1_{\mathcal{T}}\simeq GF\) from Proposition~\ref{FullyFaithfulAdjoints} and the fact that both subcategories are closed under isomorphisms.
Hence \(F(B_n)\subseteq G^{-1}(B_n)\), so  \(\left\{{F(B_n)}\right\}_{n\in\N}\leq\{G^{-1}(B_n)\}_{n\in\N}\).
\end{proof}

\begin{lemma}
\label{DualAdjointsImageLiesInPreimage}
Let \(F:\mathcal{T}\rightleftarrows \mathcal{S}:G\) be adjoint triangulated functors between triangulated categories.
Let \(\{B_n\}_{n\in\N}\) be a metric on \(\mathcal{S}\).
If $G$ is fully faithful, then \(\left\{{G(B_n)}\right\}_{n\in\N}\leq\{F^{-1}(B_n)\}_{n\in\N}\).
\end{lemma}

\begin{proof}
Dual to the proof of Lemma~\ref{AdjointsImageLiesInPreimage}.
\end{proof}

The following two lemmas help us identify which Cauchy sequences have a~trivial module colimit.
To prove Lemma~\ref{ZeroDirectLimitOfBalls}, we use an argument from \cite[Remark~3.14]{CummingsGratz24}, where it is shown that \(\moco E_*\simeq0\) if and only if \(E_*:= E_1\xrightarrow{f_{1,2}} E_2 \xrightarrow{f_{2,3}} \cdots\) is a so-called \textit{null sequence}, i.e.\ a sequence where for every \(m\in\N\) we can find \(n\geq m\) satisfying \(f_{m,n}=0\).

\begin{lemma}
\label{ZeroDirectLimitOfBalls}
Let $\mathcal{T}$ be a triangulated category with a metric \(\{B_n\}_{n\in\N}\). Let
\[
E_*:= E_1\xrightarrow{f_{1,2}} E_2 \xrightarrow{f_{2,3}} \cdots
\]
be a good Cauchy sequence. If \(\moco E_*\simeq 0\) in \(\mathfrak{L}(\mathcal{T})\), then for every \(i\in\N\) and~\(j\in\Z\) there exists \(M\in\N\) such that for all \(m\geq M\) we have \(\Sigma^j E_m\in \Summ(B_i)\).
\end{lemma}

\begin{proof}
Fix \(i\in\N\) and \(j\in\Z\). We find \(M\in\N\) such that for every \(m'\geq m\geq M\) the~triangle \(E_m\xrightarrow{f_{m,m'}} E_{m'}\rightarrow D_{m,m'}\rightarrow\Sigma E_m\) satisfies \(\Sigma^{j-1} D_{m,m'} \in B_i\).
Let us have \(m\geq M\).
Since
\(\moco E_*\simeq 0\) in \(\mathfrak{L}(\mathcal{T})\), we have, in particular, that \(\varinjlim_{n\in\N} \Hom_{\mathcal{T}}(E_m,E_n)\simeq 0\) in \(\Ab\). Then \cite[Lemma~2.2]{GobelTrlifaj12} implies that there exists some \(n\geq m\) satisfying \(f_{m,n}=0\).

Consider the triangle 
\(E_m\xrightarrow{f_{m,n}} E_{n}\rightarrow D_{m,n}\rightarrow \Sigma E_m\).
Using \(n\geq m\geq M\), we obtain \(\Sigma^{j-1} D_{m,n} \in B_i\).
Since \(f_{m,n}=0\), this triangle splits, so \(D_{m,n}\simeq E_n\oplus \Sigma E_m\).
But then
\[
\Sigma^{j-1} D_{m,n}\simeq \Sigma^{j-1} E_n\oplus \Sigma^j E_m \in B_i
.\]
Hence, we get \(\Sigma^j E_m \in \Summ(B_i)\).
\end{proof}

\begin{lemma}
\label{CharacterisationOfCompactlySupportedZeros}
Let $\mathcal{T}$ be a triangulated category with a metric \(\{B_n\}_{n\in\N}\). Let \(E_*\) be a good Cauchy sequence satisfying \(\moco E_*\in\mathfrak{WC}(\mathcal{T})\).
Then the following statements are equivalent:
\begin{enumerate}[label=(\roman*)]
\item \(\moco E_*\simeq 0\).
\item For every \(i\in\N\) and \(j\in\Z\) there exists \(M\in\N\) such that for all \(m\geq M\) we have that \(\Sigma^j E_m\in \Summ(B_i)\).
\item For every \(i\in\N\) there exists \(M\in\N\) such that for all \(m\geq M\) we have that \(E_m\in \Summ(B_i)\).
\end{enumerate}
\end{lemma}

\begin{proof}
``i) $\Rightarrow$ ii)'' By Lemma~\ref{ZeroDirectLimitOfBalls}.

``ii) $\Rightarrow$ iii)'' Condition iii) is a special case of ii) when \(j=1\).

``iii) $\Rightarrow$ i)''
Since \(\moco E_*\in\mathfrak{WC}(\mathcal{T})\), we know that there exists an index \(i\in \N\) such that \(\moco E_*\in Y(B_i)^{\perp}\).
Using the condition on \(E_*\) from the~statement of~the~lemma, we may, without loss of generality, assume that for every \(n\in\N\) it holds that \(E_n\in \Summ(B_i)\) because otherwise we can find \(M\in\N\) such that for every \(m\geq M\) it holds that \(E_m\in \Summ(B_i)\) and then cut off objects $E_k$ for small indices \(0\leq k<M\).

Let \(n\in\N\) and let \(F_n\in\mathcal{T}\) be such that \(E_n\oplus F_n\in B_i\). Then after applying Yoneda's embedding we get \(Y(E_n\oplus F_n)=Y(E_n)\oplus Y(F_n)\), so
\[
\begin{split}
0&\simeq\Hom\left(Y(E_n\oplus F_n),\moco E_*\right)\\
&\simeq\Hom\left(Y(E_n),\moco E_*\right)\oplus\Hom\left(Y(F_n),\moco E_*\right).
\end{split}
\]
In~particular, we have \(0\simeq \Hom\left(Y(E_n),\moco E_*\right)\).

Finally, we obtain
\[
\begin{split}
\End\left(\moco E_*\right)&=
\Hom\left(\varinjlim_{n\in\N} Y(E_n),\moco E_*\right)\\
&\simeq
\varprojlim_{n\in\N}\Hom\left(Y(E_n),\moco E_*\right)\simeq
\varprojlim_{n\in\N}0\simeq 0,
\end{split}
\]
so the identity on \(\moco E_*\) is equal to the zero morphism.
Hence, we get that \(\moco E_*\simeq0\).
\end{proof}

We recall a well-know characterisation of conservative triangulated functors (see for example \cite[Example~4.11]{BerghSchnurer20}).

\begin{lemma}
\label{ConservativeZeroKernel}
Let $\mathcal{T}$, \(\mathcal{S}\) be triangulated categories and let \(F:\mathcal{T}\rightarrow\mathcal{S}\) be a~triangulated functor.
Then $F$ is is conservative if and only if \(\Ker F =0\).
\end{lemma}

We prove two more technical lemmas and then we proceed with the main result of this section. 

\begin{lemma}
\label{PreimageUnderEssentiallySurjectiveFunctorCommutesWithDirectSummands}
Let \(F:\mathcal{A}\rightarrow\mathcal{B}\) be an essentially surjective additive functor between additive categories.
Let \(\mathcal{C}\subseteq\Ob(\mathcal{B})\) be a class of objects.
Then
\[
F^{-1}\big(\Summ(\mathcal{C})\big)=\Summ\big(F^{-1}(\mathcal{C})\big)
.\]
\end{lemma}

\begin{proof}
Without loss of generality assume that \(\mathcal{C}\) is closed under isomorphisms.

``$\subseteq$''
Let  \(X\in F^{-1}\big(\Summ(\mathcal{C})\big)\). Then \(F(X)\in\Summ(\mathcal{C})\), so there is some \(Y\in\mathcal{B}\) such that \(F(X)\oplus Y\in\mathcal{C}\).
Since $F$ is essentially surjective, we have \(Y\simeq F(Z)\) for~some \(Z\in\mathcal{A}\). Then \(F(X\oplus Z)\simeq F(X)\oplus Y\in\mathcal{C}\), so \(X\oplus Z\in F^{-1}(\mathcal{C})\).
Hence \(X\in\Summ\big(F^{-1}(\mathcal{C})\big)\).

``$\supseteq$''
Let \(X\in\Summ\big(F^{-1}(\mathcal{C})\big)\). Then there is \(Y\in\mathcal{A}\) such that \(X\oplus Y\in F^{-1}(\mathcal{C})\).
It follows that \(F(X)\oplus F(Y)\simeq F(X\oplus Y)\in\mathcal{C}\), so \(F(X)\in\Summ(\mathcal{C})\).
In other words, we have \(X\in F^{-1}\big(\Summ(\mathcal{C})\big)\).
\end{proof}

\begin{lemma}
\label{ConservatismWithRespectToPreimageMetric}
Let \(F:\mathcal{T}\rightleftarrows \mathcal{S}:G\) be adjoint triangulated functors between triangulated categories.
Let \(\{B_n\}_{n\in\N}\) be a metric  on \(\mathcal{T}\).
Denote the completion of \(\mathcal{T}\) with respect to \(\{B_n\}_{n\in\N}\) as \(\mathfrak{S}(\mathcal{T})\) and
the completion of \(\mathcal{S}\) with respect to~the~preimage metric \(\{G^{-1}(B_n)\}_{n\in\N}\) as \(\mathfrak{S}(\mathcal{S})\).
If $F$ is fully faithful, then the~induced functor
\(\blank\circ F:\mathfrak{S}({\mathcal{S}})\rightarrow\mathfrak{S}(\mathcal{T})\) is a well-defined conservative triangulated functor.
\end{lemma}

\begin{proof}
We know that \(\{G^{-1}(B_n)\}_{n\in\N}\) and \(\left\{{F(B_n)}\right\}_{n\in\N}\) are metrics by Lemma~\ref{PreimageIsAMetric} and~Lemma~\ref{ImageIsSometimesAMetric}.
The~functor $G$ is trivially a compression with respect to~the~metrics \(\{G^{-1}(B_n)\}_{n\in\N}\) and \(\{B_n\}_{n\in\N}\).

As $F$ is fully faithful, we have \(\left\{{F(B_n)}\right\}_{n\in\N}\leq\{G^{-1}(B_n)\}_{n\in\N}\) by Lemma~\ref{AdjointsImageLiesInPreimage}, so 
$F$ is a compression with respect to \(\{B_n\}_{n\in\N}\) and \(\{G^{-1}(B_n)\}_{n\in\N}\).
Therefore, we may use Theorem~\ref{SunZhangTheorem} to establish that the functor \(\blank\circ F:\mathfrak{S}({\mathcal{S}})\rightarrow\mathfrak{S}(\mathcal{T})\) is well-defined and triangulated.

It remains to be checked that \(\blank\circ F\) is conservative.
Let \(E\in\mathfrak{S}({\mathcal{S}})\) satisfy \(E\circ F=0\). We know that \(E\simeq\moco E_*\) for some good Cauchy sequence \(E_*\) in \(\mathcal{S}\).
By Lemma~\ref{CompressionsPreservePrecompletion}, we have \(E\circ F \simeq\moco G(E_*)\) where \(G(E_*)\) is a~good Cauchy sequence by Lemma~\ref{CompressionsPreserveCauchySequences}.
By Lemma~\ref{ZeroDirectLimitOfBalls}, the good Cauchy sequence \(G(E_*)\) satisfies that for every \(i\in\N\) there exists \(M\in\N\) such that for all \(m\geq M\) we have \(G(E_m)\in \Summ(B_i)\).

Since $G$ is essentially surjective by Proposition~\ref{FullyFaithfulAdjoints}, Lemma~\ref{PreimageUnderEssentiallySurjectiveFunctorCommutesWithDirectSummands} yields
\[
E_m\in G^{-1}\big(\Summ(B_i)\big)=\Summ\big(G^{-1}(B_i)\big)
.\]
Because \(i\in\N\) was arbitrary,
Lemma~\ref{CharacterisationOfCompactlySupportedZeros} gives us \(E\simeq0\).
We obtain \(\Ker(\blank\circ F)=0\). By Lemma~\ref{ConservativeZeroKernel}, the functor $\blank\circ F$ is conservative.
\end{proof}

\begin{theorem}
\label{CalculatingCompletionsElsewhere}
Consider a diagram of triangulated functors 
\vspace{-0.2cm}
\begin{equation}\nonumber   
\xymatrix@C=0.5cm{\mathcal{S} \ar[rrr]^{G}     &&& \mathcal{T}\ar @/_1.5pc/[lll]_{F} \ar @/^1.5pc/[lll]_{J} } 
\end{equation}
consisting of two adjoint pairs \(F:\mathcal{T}\rightleftarrows \mathcal{S}:G\) and \(G:\mathcal{S}\rightleftarrows \mathcal{T}:J\). 
Let \(\{B_n\}_{n\in\N}\) be a metric on \(\mathcal{T}\).
Denote the~completion of~\(\mathcal{T}\) with respect to \(\{B_n\}_{n\in\N}\) as \(\mathfrak{S}(\mathcal{T})\) and
the completion of \(\mathcal{S}\) with respect to~the~preimage metric \(\{G^{-1}(B_n)\}_{n\in\N}\) as \(\mathfrak{S}(\mathcal{S})\).
If $F$ and $J$ are fully faithful, then the induced functor
\(\blank\circ G:\mathfrak{S}({\mathcal{T}})\rightarrow\mathfrak{S}(\mathcal{S})\) is a~well-defined triangulated equivalence with a triangulated inverse \(\blank\circ F\).
\end{theorem}

\begin{proof}
We know that \(\{G^{-1}(B_n)\}_{n\in\N}\) is a metric by~Lemma~\ref{PreimageIsAMetric}, while \(\left\{{F(B_n)}\right\}_{n\in\N}\) and \(\left\{{J(B_n)}\right\}_{n\in\N}\) are metrics by Lemma~\ref{ImageIsSometimesAMetric}.
We need to verify that all three functors in the diagram are compressions.
The~functor $G$ is clearly a compression with respect to \(\{G^{-1}(B_n)\}_{n\in\N}\) and \(\{B_n\}_{n\in\N}\).
The functor $F$ is a compression with respect to \(\{B_n\}_{n\in\N}\) and \(\{G^{-1}(B_n)\}_{n\in\N}\) because Lemma~\ref{AdjointsImageLiesInPreimage} guarantees that \(\left\{{F(B_n)}\right\}_{n\in\N}\leq\{G^{-1}(B_n)\}_{n\in\N}\).
Symmetrically, the functor $J$ is a compression with respect to \(\{B_n\}_{n\in\N}\) and \(\{G^{-1}(B_n)\}_{n\in\N}\) since \(\left\{{J(B_n)}\right\}_{n\in\N}\leq\{G^{-1}(B_n)\}_{n\in\N}\) by Lemma~\ref{DualAdjointsImageLiesInPreimage}.

By Theorem~\ref{SunZhangTheorem}, we obtain two induced triangulated functors between the completions. Namely \(\blank\circ F:\mathfrak{S}({\mathcal{S}})\rightarrow\mathfrak{S}(\mathcal{T})\) and \(\blank\circ G:\mathfrak{S}({\mathcal{T}})\rightarrow\mathfrak{S}(\mathcal{S})\),
where \(\blank\circ F\) is a left adjoint to \(\blank\circ G\) by Corollary~\ref{InheritanceOfFullFaithfulness}. The same corollary also gives us that \(\blank\circ G\) is fully faithful.

By Lemma~\ref{ConservatismWithRespectToPreimageMetric}, the functor \(\blank\circ F\) is conservative. 
We may conclude that \(\blank\circ F\) and \(\blank\circ G\) are mutually inverse equivalences using Corollary~\ref{FullyFaithfulAdjointConservative}.
\end{proof}

\begin{remark}
\label{AnAlternativeProof}
Alternatively, a~proof of Theorem~\ref{CalculatingCompletionsElsewhere} can be obtained using completions of recollements.
For a good metric \(\{B_n\}_{n\in\N}\) on \(\mathcal{T}\), the following argument is due to Zhang.

Under the assumptions of Theorem~\ref{CalculatingCompletionsElsewhere}, there exists a recollement of triangulated categories
\begin{equation}\nonumber   
\xymatrix@C=0.5cm{\Ker G \ar[rrr]^{} &&& \mathcal{S} \ar[rrr]^{G}  \ar @/_1.5pc/[lll]_{}  \ar @/^1.5pc/[lll]_{} &&& \mathcal{T}\ar @/_1.5pc/[lll]_{F} \ar @/^1.5pc/[lll]_{J} } 
\end{equation}

The functors F, G and J are compressions by the same reasoning as in the proof of~Theorem~\ref{CalculatingCompletionsElsewhere}.
After equipping the category \(\Ker G\) with the constant good metric \(\{\Ker G\}_{n\in\N}\), all the functors of the recollement become compressions. Then \cite[Theorem~1.2]{SunZhang21} applies, and we obtain a right recollement
\begin{equation}\nonumber   
\xymatrix@C=0.5cm{\mathfrak{S}(\Ker G) \ar[rrr]^{} &&& \mathfrak{S}(\mathcal{S}) \ar[rrr]^{\blank\circ F}    \ar @/^1.5pc/[lll]_{} &&& \mathfrak{S}(\mathcal{S}) \ar @/^1.5pc/[lll]_{\blank\circ G} } 
\end{equation}
However, the completion \(\mathfrak{S}(\Ker G)\) is zero by \cite[Example~1.12]{Neeman18}. Thus \(\blank\circ F\) and \(\blank\circ G\) are mutually inverse triangulated equivalences.
\end{remark}

\begin{corollary}
Let \(\mathcal{A}\) be an abelian category with enough injectives, enough projectives, exact products, and exact coproducts (e.g.\ \(\ModR\) for a ring \(R\)).
Let \(L:\mathbf{K}(\mathcal{A})\rightarrow\mathbf{D}(\mathcal{A})\) be the localisation functor from the homotopy category of \(\mathcal{A}\) to the derived category of \(\mathcal{A}\).
Then for~every metric \(\{B_n\}_{n\in\N}\) on \(\mathbf{D}(\mathcal{A})\), there is a triangulated equivalence of~categories \(\blank\circ L:\mathfrak{S}\big(\mathbf{D}(\mathcal{A})\big)\rightarrow\mathfrak{S}\big(\mathbf{K}(\mathcal{A})\big)\) where the~completion \(\mathfrak{S}\big(\mathbf{K}(\mathcal{A})\big)\) is calculated with respect to~the~metric \(\{L^{-1}(B_n)\}_{n\in\N}\).
\end{corollary}

\begin{proof}
Under the conditions on \(\mathcal{A}\), every object in \(\HomotopyCategory(\mathcal{A})\) has both a K-injective and a K-projective resolution by \cite[Proposition~4.3.4]{Krause21}. Using \cite[Proposition~4.3.1]{Krause21}, this is equivalent to $L$ having both fully faithful right adjoint and a fully faithful left adjoint. The statement now directly follows from Theorem~\ref{CalculatingCompletionsElsewhere}.
\end{proof}

\subsection{Categories with Serre functors}
\label{sub:SerreFunctors}

In the end of this section, we we illustrate Theorem~\ref{CalculatingCompletionsElsewhere} in the context of finite dimensional algebras of finite global dimension.
Let \(A\) be a finite dimensional algebra over a~field $K$. It is well-known (see e.g.\ \cite[Theorem~6.4.13]{Krause21}) that $A$ has finite global dimension if and only the category \(\derived^b(\modf\dashmodule A)\) possesses a Serre functor \(\Serre:\derived^b(\modf\dashmodule A)\rightarrow \derived^b(\modf\dashmodule A)\).

A Serre functor (originally defined in \cite[pg.\ 519]{BondalKapranov90}) is a $K$-linear triangulated autoequivalence on a $K$-linear triangulated category \(\mathcal{T}\) such that
there is an isomorphism \(\Hom_{\mathcal{T}}(B,C)\simeq D\Hom_{\mathcal{T}}\big(C,\Serre (B)\big)\) for all \(B,C\in\mathcal{T}\) natural in both variables, where \(D=\Hom_K(\blank,K):K\dashmodule\modf\rightarrow\modf\dashmodule K\) is the standard duality.

The following argument is essentially contained in the proof of \cite[Proposition~3.6]{BondalKapranov90}.

\begin{lemma}
\label{CreatingAdjoints}
Let $K$ be a field. Let \(\mathcal{T}\) and \(\mathcal{S}\) be $K$-linear triangulated categories with Serre functors \(\Serre_{\mathcal{T}}\) and \(\Serre_{\mathcal{S}}\), respectively. Let \(F:\mathcal{T}\rightleftarrows \mathcal{S}:G\) be adjoint $K$-linear, triangulated functors. Then $F$ has a left adjoint and $G$ has a right adjoint.
\end{lemma}

\begin{proof}
Let \(B\in\mathcal{S}\) and \(C\in\mathcal{T}\). Then
\[
\begin{split}
\Hom_{\mathcal{T}}\big(G(B),C\big)&\simeq D\Hom_{\mathcal{T}}\big(C,\Serre_{\mathcal{T}}G(B)\big)\simeq D\Hom_{\mathcal{T}}\left(\Serre_{\mathcal{T}}^{-1}(C),G(B)\right)\\
&\simeq D\Hom_{\mathcal{S}}\left(F\Serre_{\mathcal{T}}^{-1}(C),B\right)\simeq D\Hom_{\mathcal{S}}\big(\Serre_{\mathcal{S}}F\Serre_{\mathcal{T}}^{-1}(C),\Serre_{\mathcal{S}}(B)\big)\\
&\simeq D\Hom_{\mathcal{S}}\big(\Serre_{\mathcal{S}}F\Serre_{\mathcal{T}}^{-1}(C),\Serre_{\mathcal{S}}(B)\big)\simeq\Hom_{\mathcal{S}}\big(B,\Serre_{\mathcal{S}}F\Serre_{\mathcal{T}}^{-1}(C)\big).
\end{split}
\]
Since all the isomorphisms above are natural, we get that \(\Serre_{\mathcal{S}}F\Serre_{\mathcal{T}}^{-1}\) is a right adjoint to $G$.

The proof that $F$ has a left adjoint is dual.
\end{proof}

\begin{remark}
\label{BonusCreatingAdjoints}
Under the notation of Lemma~\ref{CreatingAdjoints} and its proof, the right adjoint of~$G$ is equal to \(\Serre_{\mathcal{S}}F\Serre_{\mathcal{T}}^{-1}\). Since both of the Serre functors \(\Serre_{\mathcal{S}}\), \(\Serre_{\mathcal{T}}\) are autoequivalences, we get that the left adjoint $F$ to $G$ is fully faithful if and only if the right adjoint \(\Serre_{\mathcal{S}}F\Serre_{\mathcal{T}}^{-1}\) to $G$ is such. The same condition also holds for the adjoints of $F$. 
\end{remark}

Consequently, if we have a pair of adjoint functors between $K$-linear, triangulated categories with Serre functors, Lemma~\ref{CreatingAdjoints} guarantees that both of the functors from the adjoint pair have also an adjoint on the other side.
This is helpful while checking whether the conditions of Theorem~\ref{CalculatingCompletionsElsewhere} hold or not.

Since \(\derived^b(\modf\dashmodule A)\xhookrightarrow{}\derived(A)\) is a good extension whenever the algebra $A$ has finite global dimension (see~Section~\ref{sec:CertainCompletions}), we are able to take advantage of the following translation of Theorem~\ref{SunZhangTheorem} into the setting of good extensions.

\begin{proposition}
\label{InducedFunctorsInsideGoodExtensions}
Let \(\overline{F}:\overline{\mathcal{T}}\rightleftarrows \overline{\mathcal{S}}:\overline{G}\) be adjoint triangulated functors between triangulated categories with countable coproducts.
Suppose that there are good extension \(\nu_\mathcal{T}:\mathcal{T}\xhookrightarrow{}\overline{\mathcal{T}}\) and \(\nu_\mathcal{S}:\mathcal{S}\xhookrightarrow{}\overline{\mathcal{S}}\) with respect to~metrics \(\{B_n\}_{n\in\N}\) on \(\mathcal{T}\) and~\(\{C_n\}_{n\in\N}\) on \(\mathcal{S}\), respectively.
Assume that \(\overline{F}\) and \(\overline{G}\) restrict via \(\nu_\mathcal{T}\) and \(\nu_\mathcal{S}\) to~an~adjoint pair \(F:\mathcal{T}\rightleftarrows \mathcal{S}:G\), i.e.\ we have the following diagram
\[
\begin{tikzcd}
\overline{\mathcal{T}} \arrow[rr, "\overline{F}", bend right]                        &  & \overline{\mathcal{S}} \arrow[ll, "\overline{G}"', bend right]                       \\
                                                       &  &                                                        \\
\mathcal{T} \arrow[rr, "F"', bend right] \arrow[uu, "\nu_{\mathcal{T}}", hook] &  & S \arrow[ll, "G", bend right] \arrow[uu, "\nu_{\mathcal{S}}"', hook]
\end{tikzcd}
\]
where the square with \(F\) and \(\overline{F}\) commutes, and the square with \(G\) and \(\overline{G}\) commutes (both up to~natural isomorphisms).
If $F$ and $G$ are compressions and \(\overline{G}\) preserves coproducts, then we have a commutative (up to natural isomorphisms) diagram of triangulated functors
\[
\begin{tikzcd}
\mathfrak{S}'(\mathcal{T})  \arrow[dd, "\simeq", "\mathcal{Y}\restriction_{\mathfrak{S}'(\mathcal{T})}"'] &  & \mathfrak{S}'(\mathcal{S}) \arrow[ll, "\overline{G}\restriction_{\mathfrak{S}'(\mathcal{S})}"'] \arrow[dd, "\mathcal{Y}\restriction_{\mathfrak{S}'(\mathcal{S})}", "\simeq"'] \\
                      &  &                                        \\
\mathfrak{S}(\mathcal{T})                   &  & \mathfrak{S}(\mathcal{S}), \arrow[ll, "\blank\circ F"]                   
\end{tikzcd}
\]
where the equivalences \(\mathcal{Y}\restriction_{\mathfrak{S}'(\mathcal{T})}\) and \(\mathcal{Y}\restriction_{\mathfrak{S}'(\mathcal{S})}\) are given by Theorem~\ref{TheOnlyComputationalTool}.
\end{proposition}

\begin{proof}
The existence of the induced functor \(\blank\circ F\) on the completions is established by Theorem~\ref{SunZhangTheorem}.
Theorem~\ref{TheOnlyComputationalTool} guarantees that the vertical maps  \(\mathcal{Y}\restriction_{\mathfrak{S}'(\mathcal{T})}\) and \(\mathcal{Y}\restriction_{\mathfrak{S}'(\mathcal{S})}\) exist and are triangulated equivalences.

Let \(X\in\mathfrak{S}'(\mathcal{S})\). Then \(X\simeq\hoco\nu_{\mathcal{S}}(E_*)\) for some good Cauchy sequence in~\(\mathcal{S}\).
Let
\[
\coprod_{i\in\N}\nu_{\mathcal{S}}(E_i)\rightarrow\coprod_{i\in\N}\nu_{\mathcal{S}}(E_i)\rightarrow X\rightarrow\coprod_{i\in\N}\Sigma\nu_{\mathcal{S}}(E_i)
\]
be the triangle defining \(\hoco\nu_{\mathcal{S}}(E_*)\) in \(\overline{\mathcal{S}}\).
Since \(\overline{G}\) is triangulated and preserves coproducts, the triangle
\[
\coprod_{i\in\N}\overline{G}\nu_{\mathcal{S}}(E_i)\rightarrow\coprod_{i\in\N}\overline{G}\nu_{\mathcal{S}}(E_i)\rightarrow \overline{G}(X)\rightarrow\coprod_{i\in\N}\Sigma \overline{G}\nu_{\mathcal{S}}(E_i)
\]
shows that \(\overline{G}(X)=\hoco \overline{G}\nu_{\mathcal{S}}(E_*)\). 
We infer that \(\overline{G}(X)\simeq\hoco\nu_{\mathcal{T}}G(E_*)\) because of~the~commutativity rule \(\overline{G}\nu_{\mathcal{S}}\simeq\nu_{\mathcal{T}}G\).
Since \(G(E_*)\) is a Cauchy sequence by Lemma~\ref{CompressionsPreserveCauchySequences}, we get \(\overline{G}(X)\in\mathfrak{L}'(\mathcal{T})\).

Using the good extension \(\nu_{\mathcal{T}}\) we get \(\mathcal{Y}\big(\overline{G}(X)\big)\simeq\moco G(E_*)\).
Also it holds that \(\mathcal{Y}(X)\simeq\moco E_*\) since \(\nu_{\mathcal{S}}\) is a good extension.
By Lemma~\ref{CompressionsPreservePrecompletion},
\[
\mathcal{Y}(X)\circ F\simeq\moco G(E_*)\in\mathfrak{S}(\mathcal{T})\subseteq\mathfrak{C}(\mathcal{T})
.\]
This proves that \(\overline{G}(X)\in\mathcal{Y}^{-1}\big(\mathfrak{C}(\mathcal{T})\big)=\mathfrak{C}'(\mathcal{T})\), so  \(\overline{G}(X)\in\mathfrak{S}'(\mathcal{T})=\mathfrak{L}'(\mathcal{T})\cap\mathfrak{C}'(\mathcal{T})\).

We conclude that \(\overline{G}\restriction_{\mathfrak{S}'(\mathcal{S})}:\mathfrak{S}'(\mathcal{S})\rightarrow\mathfrak{S}'(\mathcal{T})\) is well defined. 
The commutativity of the~diagram follows from the already proven chain of isomorphisms
\[
\mathcal{Y}\big(\overline{G}(X)\big)\simeq\moco G(E_*)\simeq\mathcal{Y}(X)\circ F
.\]
\end{proof}

With that in mind, we can proceed to our example.

\begin{example}
Let $\mathcal{T}:=\mathbf{D}^b(\modf\dashmodule K)$ be the~bounded derived category of~finite dimensional $K$-vector spaces with the standard shift \(\Sigma\), i.e.\ the category of~finite dimensional graded vector spaces.
Consider the metric \(\{B_n\}_{n\in\N}\) given by the~cohomology functor
\(
H:\mathbf{D}^b(\modf\dashmodule K)\rightarrow\modf\dashmodule K\)
via the~formula
\[
B_n=\left\{T\in\mathbf{D}^b(\modf\dashmodule K):\forall i>-n,H^i(T)=0\right\} \text{ for $n\in\N$}
.\]
The metric \(\{B_n\}_{n\in\N}\) is an example of a metric of ``type i)'' from \cite[Example~12]{Neeman20}.
The completion with respect to \(\{B_n\}_{n\in\N}\) is triangle equivalent to \(\mathcal{T}\) perceived as a~subcategory \(\mathfrak{S}'(\mathcal{T})=\mathbf{D}^b(\modf\dashmodule K)=\mathcal{T}\)
of a good extension \(\mathbf{D}(K)\) (using Theorem~\ref{TheOnlyComputationalTool} and Theorem~\ref{PrototypicalGoodExtension}). 

Let us fix the orientation \(2\rightarrow1\) for the Dynkin quiver \(A_2\) and look at the~derived category $\mathcal{S}:=\mathbf{D}^b(\modf\dashmodule KA_2)$ with the standard shift \(\Sigma\).
The module category \(\modf\dashmodule KA_2\) contains three indecomposable modules, namely \(\simple(1)=\PP(1)=0\rightarrow K\), \(\PP(2)=\I(1)=K\xrightarrow{1}K\) and \(\I(2)=\simple(2)=K\rightarrow0\).

We have an adjoint pair of functors
\[
\blank\otimes_K^{\mathbf{L}}\PP(2):\mathbf{D}^b(\modf\dashmodule K)\rightleftarrows\mathbf{D}^b(\modf\dashmodule KA_2):\mathbf{R}\Hom_{KA_2}\big(\PP(2),\blank\big).
\]
This gives us a preimage metric
\[
\left\{\mathbf{R}\Hom_{KA_2}\big(\PP(2),\blank\big)^{-1}(B_n)\right\}_{n\in\N}
\] on~the~category \(\mathbf{D}^b(\modf\dashmodule KA_2)\), where we can calculate  
\[
\begin{split}
\mathbf{R}\Hom_{KA_2}\big(\PP(2),\blank\big)^{-1}(B_n)=\left\{X\in\mathcal{S}:\forall i>-n,H^i(X)\in\coprodf\big(\simple(1)\big)\right\}
\end{split}
\]
for \(n\in\N\).
The completion  \(\mathfrak{S}_G'(\mathcal{S})\subseteq \mathbf{D}(KA_2)\) with respect to this metric computed inside the good extension satisfies
\(\mathfrak{S}_G'(\mathcal{S})=\langle\simple(2)\rangle\) (the explicit calculation can be done using Theorem~\ref{DynkinCompletions}). 

Both the categories \(\mathcal{T}\) and \(\mathcal{S}\) have Serre functors.
The functor \(\blank\otimes_K^{\mathbf{L}}\PP(2)\) is fully faithful and its right adjoint \(\mathbf{R}\Hom_{KA_2}\big(\PP(2),\blank\big)\) has a fully faithful right adjoint by Lemma~\ref{CreatingAdjoints} and~Remark~\ref{BonusCreatingAdjoints}.
Hence Theorem~\ref{CalculatingCompletionsElsewhere} applies. After translating the~statement of the theorem for good extensions (as in Proposition~\ref{InducedFunctorsInsideGoodExtensions}) we get that the restriction 
\(\mathbf{R}\Hom_{KA_2}\big(\PP(2),\blank\big):\mathfrak{S}_G'(\mathcal{S})=\langle\simple(2)\rangle\rightarrow\mathfrak{S}'(\mathcal{T})=\mathcal{T}\) is a~triangulated equivalence of categories.

\end{example}

\section{Good metrics and improvements}
\label{sec:improvements}

The purpose of this section is to compare the properties of a good metric and~a~generic (\textit{a priori} not good) metric. The notion of good metric is, on one hand, stronger, since it requires the~additional rapid decrease condition. On the~other hand, it makes the calculation of the completion easier as described below.

We also make an argument that a general (non-good) metric can be improved to a~good metric (called improvement).
Then the completion with respect to the~original metric is a triangulated subcategory of the completion of the improvement metric.

\begin{lemma}
\label{BadIsGood}
Let $\mathcal{T}$ be a triangulated category with a~good metric \(\{B_n\}_{n\in\N}\).
Let \(E_*:= E_1\rightarrow E_2 \rightarrow \cdots\) be a directed diagram in \(\mathcal{T}\). Then \(E_*\) is a~good Cauchy sequence if and only if it is a Cauchy sequence.
\end{lemma}

\begin{proof}
``$\Rightarrow$'' Every good Cauchy sequence is a Cauchy sequence by the definition.

``$\Leftarrow$'' Fix \(i\in\N\) and \(j\in\Z\). We find \(M\in\N\) such that for any choice of numbers \(m'\geq m\geq M\) the~triangle \(E_m\rightarrow E_{m'}\rightarrow D_{m,m'}\rightarrow\Sigma E_m\) satisfies \(D_{m,m'} \in B_{i+|j|}\). Then the rapid decrease gives us \(\Sigma^{\pm1}D_{m,m'} \in B_{i+|j|-1}\), \(\Sigma^{\pm2}D_{m,m'} \in B_{i+|j|-2}\) etc.

In particular, \(\Sigma^{\pm j}D_{m,m'} \in B_{i}\).
Hence, the sequence \(E_*\) is good Cauchy.
\end{proof}

\begin{definition}
\label{DefinitionImprovement}
Let $\mathcal{T}$\, be a triangulated category with a metric \(\{B_n\}_{n\in\N}\).
We define \textit{the improvement of}~\(\{B_n\}_{n\in\N}\) as the sequence of subcategories \(\{C_n\}_{n\in\N}\) of~\(\mathcal{T}\) given inductively by the rule \(C_1:=B_1\) and
\[
C_n:=B_n\cap\Sigma C_{n-1}\cap\Sigma^{-1}C_{n-1}  \text{ for } n>1. 
\]
\end{definition}

\begin{remark}
If \(\{B_n\}_{n\in\N}\) is a good metric, it is equal to its own improvement. The~rapid decrease property of good metrics gives us \(B_n\subseteq \Sigma B_{n-1}\), \(B_n\subseteq \Sigma^{-1} B_{n-1}\), and consequently \(B_n= B_n\cap\Sigma B_{n-1}\cap\Sigma^{-1} B_{n-1}\) for \(n>1\).     
\end{remark}

\begin{lemma}
\label{ExplicitDescriptionOfImprovement}
Let $\mathcal{T}$ be a triangulated category with a metric \(\{B_n\}_{n\in\N}\) together with its improvement \(\{C_n\}_{n\in\N}\).
Then \(C_n=\bigcap_{i=-n+1}^{n-1}\Sigma^iB_{n-|i|}\) for all \(n\in\N\).
\end{lemma}

\begin{proof}
Denote \(G_n:=\bigcap_{i=-n+1}^{n-1}\Sigma^iB_{n-|i|}\) for all \(n\in\N\).    
We prove \(C_n=G_n\) by~induction on \(n\in\N\).
The case \(n=1\) is trivial.

Suppose that we already know \(C_n=G_n\) for some \(n\in\N\). Let \(X\in G_{n+1}\).
From the~induction premise and 
\(
X\in G_{n+1}=\bigcap_{i=-n}^{n}\Sigma^iB_{n+1-|i|}
,\)
we infer \(X\in B_{n+1}\),
\[
\Sigma X\in\bigcap_{i=-n+1}^{0}\Sigma^iB_{n-|i|}\cap\bigcap_{i=1}^{n-1}\Sigma^iB_{n+2-|i|}
\subseteq \bigcap_{i=-n+1}^{n-1}\Sigma^iB_{n-|i|}=G_n=C_n
,\]
and similarly \(\Sigma^{-1}X\in C_n\), so \(X\in C_{n+1}\) and \(G_{n+1}\subseteq C_{n+1}\).

Conversely, let \(X\in C_{n+1}\subseteq B_{n+1}\). Then \(\Sigma X\in C_n=G_n=\bigcap_{i=-n+1}^{n-1}\Sigma^iB_{n-|i|}\),
so we have \(X\in\bigcap_{i=-n}^{0}\Sigma^iB_{n+1-|i|}\). An analogous argument with \(\Sigma^{-1}X\) yields \(X\in\bigcap_{i=0}^{n}\Sigma^iB_{n+1-|i|}\).
Hence, we conclude that \(
X\in G_{n+1}=\bigcap_{i=-n}^{n}\Sigma^iB_{n+1-|i|}
\)
and~\(G_{n+1}=C_{n+1}\).
\end{proof}

\begin{proposition}
Let $\mathcal{T}$\, be a triangulated category with a metric \(\{B_n\}_{n\in\N}\). Let \(\{C_n\}_{n\in\N}\) be the improvement of \(\{B_n\}_{n\in\N}\). Then \(\{C_n\}_{n\in\N}\) is a good metric.    
\end{proposition}

\begin{proof}
We only need to check the closure under extensions for the other conditions of being a good metric (Defintion~\ref{DefinitionMetric}) are obvious from the definition of the improvement. 
Let \(n\in\N\).
By Lemma~\ref{ExplicitDescriptionOfImprovement}, we have \(C_n=\bigcap_{i=-n+1}^{n-1}\Sigma^iB_{n-|i|}\). This implies that \(C_n\) is extension-closed as an intersection of extension-closed subcategories of~\(\mathcal{T}\).
\end{proof}

\begin{remark}
The improvement \(\{C_n\}_{n\in\N}\) of \(\{B_n\}_{n\in\N}\) in the sense of Definition~\ref{DefinitionImprovement} is a refinement of \(\{B_n\}_{n\in\N}\) since \(C_n\subseteq B_n\) for all \(n\in\N\).   
\end{remark}

\begin{proposition}
\label{ImprovementIsNotRandom}
Let $\mathcal{T}$ be a triangulated category with a metric \(\{B_n\}_{n\in\N}\) together with its improvement \(\{C_n\}_{n\in\N}\).
Then \(\{C_n\}_{n\in\N}\) is the coarsest good metric (up to equivalence of metrics) which is finer than \(\{B_n\}_{n\in\N}\). 
\end{proposition}

\begin{proof}
Let \(\{A_n\}_{n\in\N}\) be a good metric finer than \(\{B_n\}_{n\in\N}\). We claim that it also holds that \(\{A_n\}_{n\in\N}\leq\{C_n\}_{n\in\N}\).
We show that for all \(n\in\N\) there exists \(m_n\in\N\) with \(A_{m_n}\subseteq C_n\) by induction on $n$. For the base case, we can simply find \(m_1\in\N\) such that \(A_{m_1}\subseteq B_1=C_1\). 

Now, suppose that we already have \(m_n\in\N\) with \(A_{m_n}\subseteq C_n\).
We find \(m\in\N\), \(m\geq m_n\), such that \(A_m\subseteq B_{n+1}\).
By the rapid decrease property of \(\{A_n\}_{n\in\N}\) and~the~induction premise, we have
\[
\Sigma^{-1}A_{m+1}\cup A_{m+1}\cup\Sigma A_{m+1}\subseteq A_m \subseteq A_{m_n} \subseteq C_n
.\]
Hence, it holds that
\[
A_{m+1}\subseteq\left\{T\in B_{n+1}:\Sigma^{\pm1}T\in C_n\right\}=C_{n+1}
,\]
so we may finish the proof by~setting \(m_{n+1}:=m+1\).
\end{proof}

\begin{proposition}
\label{GoodAndBadShareCauchySequences}
Let $\mathcal{T}$ be a triangulated category with a metric \(\{B_n\}_{n\in\N}\) together with its improvement \(\{C_n\}_{n\in\N}\).
Let \(E_*:= E_1\rightarrow E_2 \rightarrow \cdots\) be a directed diagram in \(\mathcal{T}\).
Then the following statements are equivalent:
\begin{enumerate}[label=(\roman*)]
    \item \(E_*\) is a good Cauchy sequence with respect to \(\{B_n\}_{n\in\N}\).
    \item \(E_*\) is a good Cauchy sequence with respect to \(\{C_n\}_{n\in\N}\).
    \item \(E_*\) is a Cauchy sequence with respect to \(\{C_n\}_{n\in\N}\).
\end{enumerate}
\end{proposition}

\begin{proof}
``i) $\Rightarrow$ iii)''
Fix \(n\in \N\), and find \(M\in\N\) such that for every \(m'\geq m\geq M\) the~triangle \(E_m\rightarrow E_{m'}\rightarrow D_{m,m'}\) satisfies
\[
\Sigma^i D_{m,m'} \in B_n \text{\ \ \ for all\ \ \ } i\in\{-n+1,-n+2,\ldots,n-2,n-1\}
.\]
Then
\[
D_{m,m'}\in\bigcap_{i=-n+1}^{n-1}\Sigma^iB_{n}\subseteq \bigcap_{i=-n+1}^{n-1}\Sigma^iB_{n-|i|} =C_n
\]
by Lemma \ref{ExplicitDescriptionOfImprovement}.
Hence \(E_*\) is Cauchy with respect to the improvement.

``ii) $\Leftrightarrow$ iii)'' By Lemma \ref{BadIsGood}.

``ii) $\Rightarrow$ i)'' Is obvious because \(C_n\subseteq B_n\) for all \(n\in\N\).
\end{proof}

\begin{corollary}
\label{ImprovementContainsOriginalCompletion}
Let $\mathcal{T}$ be a triangulated category with a metric \(\{B_n\}_{n\in\N}\) together with its improvement \(\{C_n\}_{n\in\N}\).
Denote by \(\mathfrak{L}_B(\mathcal{T})\), \(\mathfrak{C}_B(\mathcal{T})\) and \(\mathfrak{S}_B(\mathcal{T})\) the~pre\=/completion, the compactly supported elements and the completion with respect to~the~metric \(\{B_n\}_{n\in\N}\), respectively. We also introduce the analogous notation
\(\mathfrak{L}_C(\mathcal{T})\), \(\mathfrak{C}_C(\mathcal{T})\) and \(\mathfrak{S}_C(\mathcal{T})\) for \(\{C_n\}_{n\in\N}\).
Then \(\mathfrak{S}_B(\mathcal{T})\) is a triangulated subcategory of \(\mathfrak{S}_C(\mathcal{T})\).
\end{corollary}

\begin{proof}
Proposition~\ref{GoodAndBadShareCauchySequences} yields \(\mathfrak{L}_B(\mathcal{T})\)=\(\mathfrak{L}_C(\mathcal{T})\).
Since \(\{C_n\}_{n\in\N}\leq\{B_n\}_{n\in\N}\), we have \(\mathfrak{C}_B(\mathcal{T})\subseteq\mathfrak{C}_C(\mathcal{T})\), which implies \(\mathfrak{S}_B(\mathcal{T})\subseteq\mathfrak{S}_C(\mathcal{T})\).
\end{proof}

\begin{lemma}
\label{CompactlySupportedForGoodMetrics}
Let $\mathcal{T}$ be a triangulated category with shift \(\Sigma\) and a~good metric \(\{B_n\}_{n\in\N}\).
Then \(\mathfrak{C}(\mathcal{T})=\mathfrak{WC}(\mathcal{T})\).
\end{lemma}

\begin{proof}
The inclusion \(\mathfrak{C}(\mathcal{T})\subseteq\mathfrak{WC}(\mathcal{T})\) is trivial.
For the other inclusion, pick \(j\in\Z\) and \mbox{\(F\in\mathfrak{WC}(\mathcal{T})\)}. So \(F\in Y(B_i)^{\perp}\) for some \(i\in\N\).
By the goodness of \(\{B_n\}_{n\in\N}\), we have \(\Sigma^j B_{i+|j|}\subseteq B_i\).
Hence \(F\in Y\left(\Sigma^j B_{i+|j|}\right)^{\perp}\). Since $j$ was arbitrary, the proof follows.
\end{proof}

\begin{remark}
\label{GoodIsTheBest}
After combining Lemmas \ref{BadIsGood} and \ref{CompactlySupportedForGoodMetrics}, we see that if we have a~good metric, it does not matter whether we calculate the completion using good or ordinary Cauchy sequences, weakly or ordinarily compactly supported elements - the result is always the same.
\end{remark}

\subsection{Examples}
\label{sub:ExamplesImprovements}
We have shown that the completion with respect to the improvement contains the completion with respect to the original metric (Corollary~\ref{ImprovementContainsOriginalCompletion}), and that completions with respect to good metrics can be computed using only (\textit{a~priori} non-good) Cauchy sequences (Remark~\ref{GoodIsTheBest}).

This subsection consists of examples and counterexamples attesting that we cannot, in full generality, say more. The above conditions are indeed necessary.

The completions in all the examples below are calculated inside a good extension (as in Theorem~\ref{PrototypicalGoodExtension}) using Theorem~\ref{DynkinCompletions}.

\begin{example}
\label{GradedVectorSpaces}
Let $\mathcal{T}$ be a triangulated category with a metric \(\{B_n\}_{n\in\N}\) together with its improvement \(\{C_n\}_{n\in\N}\).
Thanks to Proposition~\ref{GoodAndBadShareCauchySequences} we know that both metrics yield the same good Cauchy sequences, so the pre-completion \(\mathfrak{L}(\mathcal{T})\) is the~same regardless of which of these metric we use.
However, this is not, in general, the case for the compactly supported elements.

Let $\mathcal{T}:=\mathbf{D}^b(\modf\dashmodule K)$
be the~bounded derived category of the category of~finite dimensional $K$-vector spaces, i.e.\ 
the category of~finite dimensional graded $K$-vector spaces.
Consider a metric \(\{B_n\}_{n\in\N}\) given by the~cohomology functor
\(
H:\mathbf{D}^b(\modf\dashmodule K)\rightarrow\modf\dashmodule K\)
via the formula 
\[
 B_n=\{T\in\mathbf{D}^b(\modf\dashmodule K):\forall i>-n,i\neq 0,H^i(T)=0\} \text{ for $n\in\N$}
.\]
The improvement \(\{C_n\}_{n\in\N}\) of \(\{B_n\}_{n\in\N}\) can be calculated using Lemma~\ref{ExplicitDescriptionOfImprovement}, and  we obtain
\[
C_1=B_1 \text{ and } C_n=\{T\in\mathbf{D}^b(\modf\dashmodule K):\forall i>-n,H^i(T)=0\} \text{ for $n>2$}
.\]
From the perspective of graded vector spaces, the~only difference between the~metric \(\{B_n\}_{n\in\N}\) and its improvement \(\{C_n\}_{n\in\N}\) is that \(\{B_n\}_{n\in\N}\) allows non-zero vector spaces in~degree~0.

Note that \(\{C_n\}_{n\in\N}\) is equivalent to the metric of ``type i)'' from \cite[Example~12]{Neeman20}.
The completion of \(\mathcal{T}\) with respect to \(\{C_n\}_{n\in\N}\) is equal to \(\mathcal{T}=0^{\perp}\cap\,\mathcal{T}\) by Theorem~\ref{DynkinCompletions} since \(\bigcap_{n=1}^{\infty}C_n=0\).
However, the completion of \(\mathcal{T}\) with respect to \(\{B_n\}_{n\in\N}\) is \(0=K[-\infty,\infty]^{\perp}\) by Theorem~\ref{DynkinCompletions} because \(K\in\bigcap_{n=1}^{\infty}B_n\). Indeed, no non-zero object of \(T\in\mathcal{T}\) is compactly supported with respect to \(\{B_n\}_{n\in\N}\) as one can always find \(j\in\Z\) and a non-zero morphism \(\Sigma^jK\rightarrow T\). 
\end{example}

\begin{example}
\label{GoodnessIsNecessary}
If we try to compute a ``completion'' using non-good Cauchy sequences with respect to a non-good metric, we may not get a triangulated category.

Look at the~bounded derived category $\mathcal{T}:=\mathbf{D}^b(\modf\dashmodule \Z)$  of the category of~fi\-ni\-te\-ly generated abelian groups.
Let \(\mathcal{C}\subseteq\modf\dashmodule\Z\) be the subcategory consisting of finitely generated torsion groups.
Consider the constant (non-good) metric \(\{B_n\}_{n\in\N}=\{\mathcal{C}\}_{n\in\N}\), i.e.\ 
\[
B_n=\{T\in\mathbf{D}^b(\modf\dashmodule \Z):\forall i\in\Z,i\neq0,H^i(T)=0\wedge H^0(T)\in\mathcal{C}\} \text{ for all $n\in\N$}
.\]

We make use of the good extension
\(\derived^b(\modf\dashmodule \Z)\xhookrightarrow{}\derived(\Z)\) provided by~Theorem~\ref{PrototypicalGoodExtension}.
Similarly as in Example~\ref{HereditaryNonExample}, we look at the following diagram in \(\modf\dashmodule\Z\)
\[
C_*:= \Z \xhookrightarrow{2\cdot\blank} \Z \xhookrightarrow{3\cdot\blank} \Z \xhookrightarrow{4\cdot\blank} \Z \xhookrightarrow{} \cdots
\]
with \(\varinjlim C_*=\Q\) in \(\Mod\dashmodule\Z\).
We have for all \(i\in\N\) that \(\Coker(i\cdot\blank)\in\mathcal{C}\), so if we view \(\Z\) as a complex concentrated in degree $0$ and interpret $C_*$ as a diagram in~\(\mathbf{D}^b(\modf\dashmodule \Z)\), then \(C_*\) is Cauchy (but not good Cauchy) with \(\hoco C_*=\Q\) in \(\derived(\Z)\).

Trivially, we have \(\Q\in(\Sigma^j\mathcal{C})^{\perp}\) for \(j>0\).
Because $\Q$ is torsion-free and injective, we also have \(\Q\in(\Sigma^j\mathcal{C})^{\perp}\) for \(j\leq0\).
All together, we get that \(\Q\in\mathfrak{L}'(\mathcal{T})\cap\mathfrak{C}'(\mathcal{T})\).

On the other hand, it holds that \(\Sigma\Q\notin\mathfrak{L}'(\mathcal{T})\) since for any Cauchy sequence \(E_*\) the \((-1)\)-st cohomology
\[
H^{-1}\big(\hoco E_*\big)\simeq \varinjlim H^{-1}(E_*)\in\modf\dashmodule\Z
\]
is a directed colimit of an eventually constant sequence \(H^{-1}(E_*)\) of finitely generated abelian groups. We conclude that \(\mathfrak{L}'(\mathcal{T})\cap\mathfrak{C}'(\mathcal{T})\)
is not triangulated.

We finish by observing that the property of \(\mathfrak{L}'(\mathcal{T})\cap\mathfrak{C}'(\mathcal{T})\) not being triangulated transfers to \(\mathfrak{L}(\mathcal{T})\cap\mathfrak{C}(\mathcal{T})\) as well.
Indeed, we have \(\moco C_*=\mathcal{Y}(\Q)\) by~the~goodness of the extension
\(\derived^b(\modf\dashmodule \Z)\xhookrightarrow{}\derived(\Z)\).
Since
\[
\Sigma\mathcal{Y}(\Q)=\Sigma\Hom(\blank,\Q)=\Hom(\Sigma^{-1}\blank,\Q)\simeq\Hom(\blank,\Sigma\Q)=\mathcal{Y}(\Sigma\Q)
\]
and \(\mathcal{Y}\restriction\mathfrak{L}'(\mathcal{T})\) is full and conservative by \cite[Remark~21]{Neeman18} (see also Lemma~\ref{ExtensionFakeYonedaLemma}), we obtain then any potential Cauchy sequence converging towards \(\Sigma\mathcal{Y}(\Q)\) in \(\Mod\dashmodule\mathcal{T}\) would have homotopy colimit \(\Sigma\Q\) in \(\derived(\Z)\), which is impossible.
Hence the~category \(\mathfrak{L}(\mathcal{T})\cap\mathfrak{C}(\mathcal{T})\) is not triangulated.
\end{example}

\begin{example}
If we try to compute a ``completion'' using weakly compactly supported elements, we may not get a triangulated category.   

We consider the~bounded derived category $\mathcal{T}:=\mathbf{D}^b(\modf\dashmodule K)$  of the category of graded finite dimensional vector spaces with the standard shift \(\Sigma\).
Consider a~metric \(\{B_n\}_{n\in\N}\) given by the~cohomology functor
\(
H:\mathbf{D}^b(\modf\dashmodule K)\rightarrow\modf\dashmodule K\)
via the formula
\[
B_n=\{T\in\mathbf{D}^b(\modf\dashmodule K):\forall i\in\Z,i\neq0,H^i(T)=0\} \text{ for } n\in\N
.\]

We make use of the good extension
\(\derived^b(\modf\dashmodule K)\xhookrightarrow{}\derived(K)\) and we view ordinary vector spaces from \(\Mod\dashmodule K\) as complexes concentrated in degree $0$.
Since it holds that \(K\in B_n\) for all \(n\in\N\), we have \(K\notin\mathfrak{WC}'(\mathcal{T})\).
However \(\Sigma K\in\mathfrak{L}'(\mathcal{T})\cap\mathfrak{WC}'(\mathcal{T})\), so the~subcategory is not closed under the shift, nor triangulated.

We may than infer that \(\mathfrak{L}(\mathcal{T})\cap\mathfrak{WC}(\mathcal{T})\) is not triangulated using the same argument as in Example~\ref{GoodnessIsNecessary}.
\end{example}

\ 

\subsubsection*{Conflict of Interests} The author has no relevant financial or non-financial interests to disclose.

\bibliographystyle{plain}
\bibliography{literatura}

\end{document}